\RequirePackage{fix-cm}
\documentclass[a4paper]{elsarticle}

\makeatletter
\def\ps@pprintTitle{%
 \let\@oddhead\@empty
 \let\@evenhead\@empty
 \def\@oddfoot{}%
 \let\@evenfoot\@oddfoot}
\makeatother

\usepackage[margin=1in]{geometry}
\usepackage{amssymb, amsthm, amsmath, amsfonts}
\usepackage{xfrac}
\usepackage{bm}
\usepackage{enumitem}
\usepackage[usenames,dvipsnames]{xcolor}
\usepackage{mathtools}
\usepackage{appendix}

\usepackage{thmtools}

\theoremstyle{plain}
\declaretheorem[name=Proposition, numberwithin=section]{proposition}

\declaretheorem[name=Lemma, sibling=proposition]{lemma}

\theoremstyle{definition}
\declaretheorem[name=Remark, qed={$\triangle$}, sibling=proposition]{remark}
\declaretheorem[name=Example, qed = {$\bigcirc$}, sibling=proposition]{example}
\declaretheorem[name=Definition, sibling=proposition]{definition}
\declaretheorem[name=Assumption, sibling=proposition]{assumption}
\declaretheorem[name=Notation,numbered=no]{notation}

\usepackage[colorlinks]{hyperref}

\usepackage{mathrsfs}
\DeclareFontFamily{U}{rsfs}{\skewchar\font127}
\DeclareFontShape{U}{rsfs}{m}{n}{<-6> rsfs5 <6-8> rsfs7 <8-> rsfs10}{}

\newcommand{\numberthis}{\addtocounter{equation}{1}\tag{\theequation}}

\DeclareMathOperator{\R}{\mathbb{R}}
\DeclareMathOperator*{\E}{\mathbb{E}}
\renewcommand*{\d}{\,\mathrm{d}}
\DeclareMathOperator*{\osc}{\mathrm{osc}\,}
\DeclareMathOperator*{\Dom}{\mathrm{Dom}\,}

\begin{document}

\begin{frontmatter}
	\title{Bilinear equations in Hilbert space driven by paths of low regularity}
	\author[MFF]{Petr \v{C}oupek}
	\ead{coupek@karlin.mff.cuni.cz}
	\author[Sev]{Mar\'ia J. Garrido-Atienza\corref{cor}}
	\ead{mgarrido@us.es}
	
	\cortext[cor]{Corresponding author}
	\address[MFF]{Charles University, Faculty of Mathematics and Physics, Sokolovsk\'a 83, Prague 8, 186 75, Czech Republic}
	\address[Sev]{Universidad de Sevilla, Dpto. Ecuaciones Diferenciales y An\'alisis num\'erico, Avda. Reina Mercedes s/n, 41012-Sevilla, Spain}

	\begin{keyword}
	It\^o-F\"ollmer calculus, $p$-th variation along a sequence of partitions, Multiplicative noise, Fractional Brownian motion, Rough paths.
	\MSC[2010] Primary: 60H15, 60G22; Secondary: 34F05, 47D06.
	\end{keyword}

	\begin{abstract}
		In the article, some bilinear evolution equations in Hilbert space driven by paths of low regularity are considered and solved explicitly. The driving paths are scalar-valued and continuous, and they are assumed to have a finite $p$-th variation along a given sequence of partitions in the sense given by Cont and Perkowski \cite{ConPer18} ($p$ being an even positive integer). Typical functions that satisfy this condition are trajectories of the fractional Brownian motion of the Hurst parameter $H=\sfrac{1}{p}$. A strong solution to the bilinear problem is shown to exist if there is a solution to a certain temporally inhomogeneous initial value problem. Subsequently, sufficient conditions for the existence of the solution to this initial value problem are given. The abstract results are applied to several stochastic partial differential equations with multiplicative fractional noise, both of the parabolic and hyperbolic type, that are solved explicitly in a pathwise sense. 
	\end{abstract}
\end{frontmatter}

\section{Introduction}

\setlength{\parindent}{0em}

In the article, the evolution equation
	\begin{equation}
	\label{eq:intro_1}
		\begin{cases}
			\d{X}_t& = AX_t\d{t} + BX_t\d^\pi\omega_t, \quad s<t\leq T,\\
			\phantom{\d}X_s & = x_0,
		\end{cases}
	\end{equation}
in a Hilbert space $V$ is studied. Here, $A:V\supseteq\Dom A\rightarrow V$ and $B:V\supseteq \Dom B\rightarrow V$ are (possibly unbounded) linear operators and $\omega: [0,T]\rightarrow \R$ is a continuous function that has finite $p$-th variation along a sequence of partitions $\{\pi_n\}_{n\in\mathbb{N}}$ of the interval $[0,T]$ in the sense of \cite{ConPer18} for some positive even integer $p$. This means that the sequence of measures $\{\mu_n\}_{n\in\mathbb{N}}$ on the measurable space $([0,T],\mathcal{B}([0,T]))$ defined by 
	\begin{equation*}
		\mu_n := \sum_{[t_j, t_{j+1}]\in\pi_n} \delta_{t_j} |\omega_{t_{j+1}}-\omega_{t_j}|^p
	\end{equation*}
converges weakly to a non-atomic measure $\mu$ ($\delta_u$ denotes the Dirac measure at the point $u\in\R$). The $p$-th variation is then defined as the function $[\omega]_p^\pi(t):=\mu ([0,t])$. The notion of $p$-th variation along a sequence of partitions first appeared in F\"ollmer's seminal paper \cite{Fol81} where the case $p=2$ is considered. The concept is generalized to $p>2$ by Cont and Perkowski in the paper \cite{ConPer18}.

\medskip

Functions that have a finite $p$-th variation along a sequence of partitions very often arise as trajectories of stochastic processes. For example, it is well-known that almost every path of the Wiener process $(W_t)_{t\in [0,1]}$ has finite quadratic variation $[W]_2^\pi(t)=t$ along every non-random sequence of partitions $\pi=\{\pi_n\}_{n\in\mathbb{N}}$ of the interval $[0,T]$ whose mesh tends to zero, see the remark on p.~149 of \cite{Fol81} or the remark following Theorem 7.4 in the recent paper \cite{DavOblSio18}. However, it is also possible to consider processes of lower regularity. For example, almost every path of the fractional Brownian motion $W^H$ with the Hurst parameter $H=\sfrac{1}{2k}$, $k\in\mathbb{N}$, has a finite $2k$-th variation along a subsequence of the sequence $E_n:=\{\sfrac{i}{n}\,|\,i=0,1,\ldots,n\}$ (this in fact holds for any $H\in (0,1)$ with $p=\sfrac{1}{H}$, see e.g. \cite{EssNua15, GueNua05, Prat11, Rog97, Tud13}, but the non-integer case is not considered in this paper).

\medskip

The term $BX_t\d^\pi\omega$ in equation \eqref{eq:intro_1} formally corresponds to the linear multiplicative noise term $BX_t\,\dot{\omega}$. It is natural to interpret equation \eqref{eq:intro_1} as the integral equation
	\begin{equation}
	\label{eq:intro_integra_eq}
		X_t = x_0 + \int_s^t AX_r\d{r} + \int_s^t BX_r\d^\pi\omega_r, \quad s< t\leq T;
	\end{equation}
however, the path $\omega$ is not of bounded variation (as $p$ is assumed to be larger than one) and therefore, the second integral above has to be given a rigorous meaning. It is a key observation of F\"ollmer in \cite{Fol81} that the integral can be defined with respect to a path $\omega$ that has a finite quadratic variation along a sequence of partitions $\pi$ as a pointwise limit of Riemann sums for the case when the integrand is a $\mathscr{C}^2$ functional of the path itself. As shown by Cont and Perkowski in \cite{ConPer18}, this remains true even when $\omega$ is a path of finite $p$-th variation with $p$ being a positive even integer and the functional is of the class $\mathscr{C}^p$. 

\medskip

Following \cite{ConPer18, Fol81}, in the setting of the present paper, a function $h:[0,T]\rightarrow V$ is said to be integrable with respect to the path $\omega$ if there exists a function $f\in \mathscr{C}^{1,p}([0,T]\times\R;V)$ such that $h_t=(\partial_2f)(t,\omega_t)$ for every $t\in [0,T]$ and its integral is defined by the limit (in the topology of the space $V$) of the compensated Riemann sums
	\begin{equation}
	\label{eq:intro_2}
		\int_0^t h_r\d^\pi\omega_r := \lim_{n\rightarrow \infty} \sum_{[t_j, t_{j+1}]\in\pi_n}\sum_{k=1}^{p-1}(\partial_2^kf)(t_j,\omega_{t_j})\frac{1}{k!}(\omega_{t_{j+1}\wedge t}-\omega_{t_j\wedge t})^k, \quad 0\leq t\leq T.
	\end{equation}
As shown in \autoref{lem:change_of_variable} below, this limit exists and its value can be explicitly computed via a certain change-of-variable formula. \autoref{lem:change_of_variable} is a modification of \cite[Theorem 1.5 and Theorem 1.10]{ConPer18} that is more suitable for our purposes.

\medskip

Having given a rigorous meaning to equation \eqref{eq:intro_1}, a main result of the present paper is that if $v_{s,x_0}$ is a solution to the temporally inhomogeneous initial value problem 
	\begin{equation}
	\label{eq:intro_3}
		\begin{cases}
			\dot{v}(t) & = \tilde{C}_s(t)v(t), \quad s<t\leq T,\\
			v(s) & = x_0, 
		\end{cases}
	\end{equation}
for $s\in [0,T)$ and a sufficiently regular initial state $x_0\in V$ where $\tilde{C}_s$ is the family of linear operators defined by 
	\begin{equation*}
		\tilde{C}_s(t):=G_B(\omega_t-\omega_s)\left(A-\frac{1}{p!}\dot{[\omega]}^\pi_p(t)B^p\right)G_B(\omega_s-\omega_t), \quad 0\leq t\leq T,
	\end{equation*}
with $G_B$ being the strongly continuous group of linear operators generated by $B$ and $\dot{[\omega]}_p^\pi$ being the derivative of the $p$-th variation of $\omega$ along $\pi$ (that is assumed to exist), then the path $X_{s,x_0}^\omega$ defined by 
	\begin{equation*}
		X_{s,x_0}^\omega(t) := G_B(\omega_t-\omega_s)v_{s,x_0}(t), \quad s\leq t\leq T,
	\end{equation*}
is a solution to the bilinear problem \eqref{eq:intro_1} (see \autoref{prop:linear_equation_without_commutativity} for the precise statement). It therefore follows that in order to find a solution to \eqref{eq:intro_1}, it suffices to find a solution to problem \eqref{eq:intro_3}. This is done in \autoref{prop:parabolic_noncom} and \autoref{prop:hyperbolic_noncom} where two sets of sufficient conditions for the family $\tilde{C}_s$ to generate a strongly continuous evolution system of operators are given (see e.g. \cite{Pazy} or \cite{Tan79}). These two sets of conditions are usually referred to as the parabolic and the hyperbolic case and, roughly speaking, they correspond to sectoriality of the operator $\tilde{C}_s(t)$ for a fixed $t$. It should be noted however that even though the solution to the initial value problem \eqref{eq:intro_3} found in \autoref{prop:parabolic_noncom} and \autoref{prop:hyperbolic_noncom} is unique, uniqueness of the solution to the bilinear equation \eqref{eq:intro_1} is still an open problem.

\medskip

\textit{Literature and related approaches.} Stochastic evolution equations with multiplicative noise have been investigated by many authors in various settings. A main source of inspiration for the work in the present paper is the article \cite{DaPIanTub82} where a correspondence between a linear stochastic evolution equation with linear multiplicative scalar white noise and a random evolution equation was discovered. In fact, it is a main purpose of the present paper to show that the technique from \cite{DaPIanTub82} can be employed even if the stochastic bilinear equation is reinterpreted in the setting of the It\^o-F\"ollmer calculus. Moreover, it turns out that by considering the $p$-th variation instead of the quadratic variation, a much larger variety of driving processes can be considered.

\medskip

Explicit solutions to stochastic bilinear evolution equations driven by scalar processes that are different from the Wiener process have already been given in the literature. For example, in the articles \cite{DunMasDun05} and \cite{MasSnup18}, the driving process is a fractional Brownian motion with the Hurst parameter $H>\sfrac{1}{2}$ and the stochastic integral is interpreted in the Skorokhod sense, i.e. as the Skorokhod integral composed with a suitable fractional transfer operator, see e.g. \cite{AlosNua03}. Moreover, in the case of $H>\sfrac{1}{2}$, the stochastic integral can also be defined pathwise due to high regularity of the driving path as a generalized Young integral that is given in terms of fractional derivatives, see e.g. \cite{Zahle99}, and stochastic bilinear problems with this pathwise interpretation are studied in \cite{GarMasSnup16}. On the other hand, the literature on stochastic bilinear problems driven by fractional Brownian motions with the Hurst parameter $H<\sfrac{1}{2}$ is much more scarce. In this case, explicit solutions to the bilinear problem interpreted in the Skorokhod sense are analysed in \cite{Snup10} but the pathwise approach of \cite{GarMasSnup16} cannot be used.

\medskip

The problem of existence and uniqueness of pathwise solutions to more general (semi-)linear stochastic evolution equations with general multiplicative noise of the type $G(X_t)\d{Z}_t$ has also been treated in several papers. For example, in the article \cite{MasNua03}, the driving noise $\d{Z}$ is fractional with $H>\sfrac{1}{2}$ in time and correlated in space and the authors prove existence and uniqueness of a pathwise mild solution by employing the generalized Young integration theory of \cite{Zahle99}. The case $H<\sfrac{1}{2}$ has been also treated. Specifically, in \cite{HuNu09}, the authors considered the notion of compensated fractional derivatives and investigated the existence of solutions to stochastic differential equations driven by a scalar $\beta$-H\"older continuous path $Z$ with $\beta\in (\sfrac{1}{3},\sfrac{1}{2})$ by transforming the original equation into a system of two equations, one for the path component and the other for the area component, which are connected through the algebraic Chen relation. The results in \cite{HuNu09} were extended to evolution equations in \cite{GLSch1} and \cite{GLSch2} which can be applied to the situation where $Z$ is a Hilbert-space-valued fractional Brownian motion with the Hurst index $H\in (\sfrac{1}{3},\sfrac{1}{2}]$. See also the recent paper \cite{HesNea19} that combines the rough evolution equation approach of \cite{DeyGubTin12,GubLejTin06, GubTin10} with the techniques of \cite{GLSch2}. Generally speaking, rough path theory can be applied to solve stochastic differential equations path by path and the interested reader is referred to, for example, the excellent monographs \cite{FriHai14} and \cite{FV} and the references therein.

\medskip

In the present article, instead of the rough path theory, F\"ollmer's pathwise stochastic calculus (or more precisely its generalization due to \cite{ConPer18}) is used to find explicit pathwise solutions to the bilinear equation \eqref{eq:intro_1} but it is well-known that these two approaches are closely connected. In particular, it follows by \cite[Lemma 4.7]{ConPer18}, that the $p$-tuple $\mathbb{X}=(\mathbb{X}^0,\mathbb{X}^1, \ldots, \mathbb{X}^p)$ that is given by 
	\begin{align*}
		\mathbb{X}^0_{s,t} & := 1,\\
		\mathbb{X}^k_{s,t} & := \frac{1}{k!} (\omega_{t}-\omega_s)^k, \quad k=1,2,\ldots, p-1,\\
		\mathbb{X}^{p}_{s,t} & := \frac{1}{p!} (\omega_t-\omega_s)^{p} - \frac{1}{p!} \left([\omega]^\pi_p(t) - [\omega]^\pi_p(s)\right),
	\end{align*}
is a reduced rough path of finite $p$-th variation (as defined in \cite[Definition 4.6]{ConPer18}). If it is assumed that $f$ is a real-valued function on only one variable, i.e. $f(t,x)\equiv f(x)$, then it follows by \cite[Corollary 4.11]{ConPer18} that the integral defined by formula \eqref{eq:intro_2} coincides with the rough path integral with respect to $\mathbb{X}$ (whose definition can be found in \cite[Proposition 4.10]{ConPer18}). The interested reader can find a very thorough discussion of this connection in \cite[section 5.3]{FriHai14} for $p=2$ and in \cite[section 4.2]{ConPer18} for $p\geq 2$. It should be noted however that as opposed to the rough path theory, the construction of the integral in \eqref{eq:intro_2} and the subsequent use of F\"ollmer's change-of-variable formula is relatively simple.

\medskip

F\"ollmer's pathwise calculus has been developed by many authors, see e.g. \cite{AnaCon16, ConDas19, ConFour10, DavOblSio18, Fol81, Lem83, PerPro16, Wue80}. Moreover, this calculus has already been used to solve some differential equations that are usually considered in the framework of It\^o's integration theory without relying on any probabilistic structure, see \cite{Hir19}, and its importance has been recognized in mathematical finance see e.g. \cite{DavOblRav14, FolSchi13, LochPerPro18, PerPro16} and the references therein. In all these reference, the case $p=2$ is considered. For $p>2$, see the very recent papers \cite{ConPer18, Kim19, SchiZha19}.

\medskip

\textit{Organization of the paper}. In \autoref{sec:prelim}, the notion of the $p$-th variation along a sequence of partitions is recalled (\autoref{def:pth_variation}) and two examples are given. Moreover, the integral with respect to paths with finite $p$-th variation along a sequence of partitions is defined via a change-of-variable formula (\autoref{lem:change_of_variable} and \autoref{def:admissible_integrands}). 

\medskip

The bilinear equation \eqref{eq:intro_1} is treated in \autoref{sec:bilinear_equation}. Initially, the notion of a solution is defined and then, the reasoning is split into two parts - the first part in \autoref{sec:comm} contains the important case when $A$ and $G_B$ commute on a suitable domain and the second part in \autoref{sec:non-comm} contains the general non-commutative case. 

\medskip

The main result of \autoref{sec:comm} is \autoref{prop:linear_equation}. This is followed by three examples of stochastic (partial) differential equations with a multiplicative singular fractional noise. In \autoref{ex:one-dim_eq}, the equation is studied in dimension one. In particular, the example shows that the solution to the bilinear problem \eqref{eq:intro_1} can be viewed as a generalized geometric fractional Brownian motion. The example is then extended in \autoref{ex:inf-dim_eq_bdd_B} to infinite dimensions by assuming that the operator $A$ generates an analytic semigroup and $B$ is essentially the identity operator and the example of a heat equation is given. Finally, the subsection is concluded by \autoref{ex:B_derivative} where both $A$ and $B$ are unbounded operators. 

\medskip 

The main result of \autoref{sec:non-comm} is \autoref{prop:linear_equation_without_commutativity} which links problem \eqref{eq:intro_1} to problem \eqref{eq:intro_3} without assuming that the operator $A$ commutes with the group $G_B$ and thus generalizes \autoref{prop:linear_equation}. Then, the two cases when the family of operators $(\tilde{C}_s(t), t\in [0,T])$ is parabolic and hyperbolic are treated separately in \autoref{prop:parabolic_noncom} and \autoref{prop:hyperbolic_noncom} where two sets of sufficient conditions for the existence of the unique solution to problem \eqref{eq:intro_3} are given. Each of these two results are followed by an example. In particular, \autoref{ex:parabolic} features a heat-type equation while in \autoref{ex:hyperbolic}, a Schr\"odinger-type equation is given.

\section{Preliminaries}
\label{sec:prelim}

In the first section, we recall the notion of the $p$-th variation of a continuous function along a sequence of partitions and some of its properties. A change of variable formula that is the central tool used in the following sections is also given. 

\begin{notation}
Let $-\infty<a<b<\infty$. The symbol $\mathcal{P}[a,b]$ denotes the set of finite partitions $P=\{t_0,t_1, \ldots, t_{N(P)}\}$, $N(P)\in\mathbb{N}$, of the interval $[a,b]$ such that $a=t_0<t_1<\ldots <t_{N(P)-1}<t_{N(P)}=b$. The \textit{oscillation of a function $f\in\mathscr{C}([a,b])$ along the partition $P\in\mathcal{P}[a,b]$} is defined by 
	\begin{equation*}
		\osc(f,P) := \max_{[t_j,t_{j+1}]\in P}\max_{x,y\in [t_j,t_{j+1}]} |f(x)-f(y)|
	\end{equation*}
where $[t_j,t_{j+1}]\in P$ means that both $t_j$ and $t_{j+1}$ belong to $P$ and they are immediate successors. 
\end{notation}

Recall the definition of $p$-th variation along a sequence of partitions that is given in \cite[Definition 1.1]{ConPer18}.

\begin{definition}
\label{def:pth_variation}
Let $p>0$ and $0\leq a<b<\infty$ and let $\pi=\{\pi_n\}_{n\in\mathbb{N}}\subset \mathcal{P}[a,b]$ be a sequence of partitions of the interval $[a,b]$. A function $S\in\mathscr{C}([a,b])$ is said to have a \textit{finite $p$-th variation along the sequence $\pi$} if 
	\begin{equation*}
		\lim_{n\rightarrow\infty}\osc(S,\pi_n)=0,
	\end{equation*}
and if the sequence of measures $\{\mu_n\}_{n\in\mathbb{N}}$ on the measurable space $([a,b],\mathcal{B}([a,b]))$ that is given by 
	\begin{equation*}
		\mu_n := \sum_{[t_j,t_{j+1}]\in\pi_n} \delta_{t_j}\left|S(t_{j+1})-S({t_{j})}\right|^p
	\end{equation*}
converges weakly to a non-atomic measure $\mu$. Here, $\delta_u$ denotes the Dirac measure at the point $u\in\R$. If $S$ has a finite $p$-th variation along the sequence $\pi$, the notation $S\in V_{p}(\pi)$ is used and the function $[S]^\pi_p:[a,b]\rightarrow [0,\infty)$ defined by $[S]^\pi_p(t):= \mu([a,t])$ is called the \textit{$p$-th variation of $S$ along $\pi$}.
\end{definition}

\begin{remark}
The concept of $p$-th variation defined above is extremely dependent on the sequence of partitions. In the case $p=2$, that is usually called the \textit{quadratic variation}, this phenomenon is thoroughly investigated in \cite[section 7]{DavOblSio18} where the following result is given in Theorem 7.4 (see also Scholium A.1). Let $(\Omega,\mathcal{F},\mathbb{P})$ be a probability space. It is understood that all the following random processes and variables are defined on this space. Let $(W_t, t\in[0,1])$ be a Wiener process and let $(\mathcal{F}_t)_{t\in [0,T]}$ be the filtration generated by $W$. Then for every continuous non-negative increasing stochastic process $(A_t, t\in [0,1])$ that starts at zero, there exists a sequence $\pi=\{\pi_n\}_n$ of refining random partitions of the interval $[0,1]$ such that $\{k2^{-N}\,|\,k=0,1,\ldots, 2^n\}\subseteq \pi_n$ holds for every $n\in\mathbb{N}$ $\mathbb{P}$-almost surely and for which $W\in V_{2}(\pi)$ $\mathbb{P}$-almost surely with $[W]^\pi_2(t) = A_t,$ $t\in [0,1]$. In other words, any continuous non-negative increasing stochastic process can be viewed as a quadratic variation of the Wiener process along a suitable sequence of random partitions. On the other hand, as long as we restrict ourselves to sequences of partitions that consist of stopping times, the quadratic variation along such sequences is independent of the chosen partition. More precisely, it follows by \cite[Proposition 2.3]{DavOblSio18} that if $\tau=\{\tau_n\}_n$ is a sequence of random partitions that consist of $\mathcal{F}_t^W$-stopping times and that satisfies
$\lim_{n\rightarrow\infty}\mathrm{osc}\,(W,\tau_n)=0$ $\mathbb{P}$-almost surely, then there exists a subsequence $\tilde{\tau}\subseteq\tau$ such that $W\in V_2(\tilde{\tau})$ with $[W]_2^{\tilde{\tau}}(t)=t$,  $t\in [0,1],$ $\mathbb{P}$-almost surely.
\end{remark}

\begin{remark}
Let us also stress that the concept of $p$-th variation along a sequence is very different from the usual notion of $p$-variation. In particular, it is well-known (see e.g. \cite[Remark on p. 28]{RevYor99}) that 
			\begin{equation*}
				\|W\|_{2-\mathrm{var}, [0,1]} := \left(\sup_{\pi\in\mathcal{P}[0,1]}\sum_{[t_i,t_{i+1}]\in\pi}|W_{t_{i+1}}-W_{t_i}|^2\right)^\frac{1}{2} = \infty				
			\end{equation*}
		holds $\mathbb{P}$-almost surely while, given a refining deterministic sequence of partitions $\pi=\{\pi_n\}_{n\in\mathbb{N}}$, there exists a subsequence $\tilde{\pi}\subseteq\pi$ such that $W$ has a finite quadratic variation along $\tilde{\pi}$ $\mathbb{P}$-almost surely.
\end{remark}

\begin{example}
This example, that can be found in \cite{MisSchi19}, shows that not only sample paths of random processes have finite $p$-th variation, but there are also some purely deterministic functions with this property. Define the \textit{Faber-Schauder} functions by 
	\begin{align*}
		e_{0,0}(t) &:= \max\{0,\min\{t,1-t\}\}\\
		e_{n,k}(t) &:= 2^{-\frac{n}{2}}e_{0,0}(2^nt-k), \quad n\in\mathbb{N}, k\in\mathbb{Z},
	\end{align*}
for $t\in\R$ and consider the \textit{Takagi-Landsberg function} $\tau_{H}$ with $H\in (0,1)$ defined by
	\begin{equation*}
		\tau_{H}(t):= \sum_{n=0}^\infty 2^{n\left(\frac{1}{2}-H\right)} \sum_{k=0}^{2^n-1}e_{n,k}(t), \quad 0\leq t\leq 1.
	\end{equation*}
Then by \cite[Theorem 2.1]{MisSchi19}, the function $\tau_{H}$ has a finite $\sfrac{1}{H}$-th variation along the sequence of dyadic partitions of the interval $[0,1]$, i.e. along the sequence of partitions $D=\{D_n\}_{n\in\mathbb{N}_0}$ where 
	\begin{equation}
	\label{eq:dyadic_partitions}
		D_n:=\{k2^{-n}\, \left|\right.\, k=0,1,\ldots, 2^n\}
	\end{equation}
with $[\tau_H]_\frac{1}{H}^D(t)=t\E|Z_H|^\frac{1}{H}$ for $0\leq t\leq 1$ where the random variable $Z_H$ is defined by 
	\begin{equation*}
		Z_{H} := \sum_{n=0}^\infty 2^{-n(1-H)} Y_n
	\end{equation*}
for a sequence $\{Y_n\}_{n\in\mathbb{N}_0}$ that consists of independent, identically distributed random variables with the discrete uniform distribution on the set $\{-1,1\}$ defined on some probability space $(\Omega,\mathcal{F},\mathbb{P})$.
\end{example}

\begin{example}
\label{ex:stoch_int_fBm}
Consider the fractional Brownian motion $W^H$ on the interval $[0,1]$ defined on a probability space $(\Omega,\mathcal{F},\mathbb{P})$. Assume that the Hurst parameter $H$ belongs to the interval $(0,\sfrac{1}{2}]$ and that the $\sigma$-algebra $\mathcal{F}$ is generated by the process $W^H$. 
Let $\mathscr{H}$ be the reproducing kernel Hilbert space for the process $W^H$; that is, $\mathscr{H}$ is the completion of the space of step functions on the interval $[0,1]$ with respect to the inner product
	\begin{equation*}
		\langle\bm{1}_{[0,s]},\bm{1}_{[0,s]}\rangle_{\mathcal{H}} := R_H(s,t), \quad s,t\in [0,1],
	\end{equation*}
where $R_H$ is the covariance function of the fractional Brownian motion $W^H$ given by 
	\begin{equation*}
		R_H(s,t)= \frac{1}{2}\left(s^{2H} + t^{2H}-|s-t|^{2H}\right), \quad s,t\in [0,1].  
	\end{equation*}
It can be shown that the space $\mathscr{H}$ consists of real-valued functions (since it is assumed that $H\leq\sfrac{1}{2}$) and there exists its characterisation as an image of the space $L^2(0,1)$ under a certain fractional integral, see \cite[Theorem 3.3]{DecrUstu99} and \cite[section 4]{PipTaqq01} for the details. 
On the space $\mathscr{H}$, an isonormal Gaussian process, denoted again by $W^H$, is given by the Wiener integral with respect to the process $W^H$ that is defined defined as an extension of the map $\bm{1}_{[0,t]}\mapsto W^H_t$. This means that a notion of a stochastic integral as well as that of a stochastic derivative for $W^H$ can be defined using the Malliavin calculus, see \cite{Nua06}. In particular, the \textit{stochastic derivative $D^{H}$ (with respect to $W^H$)} is defined by
	\begin{equation*}
		D^HF = \sum_{i=1}^n (\partial_if)(W^H(h_1), W^H(h_2), \ldots, W^H(h_n))h_i
	\end{equation*}
for a random variable $F$ that is given by $F=f(W^H(h_1), W^H(h_2), \ldots, W^H(h_n))$ for some $n\in\mathbb{N}$, $\{h_i\}_{i=1}^n\subset\mathscr{H}$, and $f:\mathbb{R}^n\rightarrow\mathbb{R}$ that is infinitely differentiable and such that $f$ and all its partial derivatives are bounded (the space of such $f$ is denoted by $\mathcal{S}_b$ in the sequel). It turns out that the operator $D^H:\mathcal{S}_b\rightarrow L^2(\Omega;\mathscr{H})$ is closable and its domain is the \textit{Sobolev-Watanabe space (with respect to $W^H$)} $\mathbb{D}^{1,2}_H$ that is defined as the closure of $\mathcal{S}_b$ with respect to the norm
	\begin{equation*}
		\|F\|_{\mathbb{D}_H^{1,2}} := \left(\mathbb{E}|F|^2 + \mathbb{E}\|D^HF\|_{\mathscr{H}}\right)^{\frac{1}{2}}, \quad F\in\mathcal{S}_b.
	\end{equation*}
Similarly as in the above case of real-valued random variables $F$, this space as well as the stochastic derivative can be defined for $\mathscr{H}$-valued random variables (or even general Hilbert-space-valued random variables) in which case, the space is denoted by $\mathbb{D}^{1,2}_H(\mathscr{H})$. The stochastic integral is defined as the $L^2$-adjoint of the stochastic derivative. More precisely, its domain $\Dom \delta^H$ is the set of all $u\in L^2(\Omega;\mathscr{H})$ for which there exists a constant $c>0$ such that the inequality 
	\begin{equation*}
		\left|\langle D^HF,u\rangle_{L^2(\Omega;\mathscr{H})}\right|\leq c \|F\|_{L^2(\Omega)}
	\end{equation*}
is satisfied for every $F\in\mathcal{S}_b$. For $u\in\Dom \delta^H$, the symbol $\delta^H(u)$ denotes the unique element of the space $L^2(\Omega)$, whose existence is ensured by the Riesz representation theorem, for which the equality
	\begin{equation*}
		\langle F,\delta^H(u)\rangle_{L^2(\Omega)} = \langle u,D^HF\rangle_{L^2(\Omega;\mathscr{H})}
	\end{equation*}
is satisfied for every $F\in\mathcal{S}_b$. The operator $\delta^H:\Dom\delta^H\rightarrow L^2(\Omega)$ is then called the \textit{stochastic integral (with respect to $W^H$)}. 

Assume now that $u\in \mathbb{D}^{1,2}(\mathscr{H})$ is such that there exists $q>\sfrac{1}{H}$ and constants $L_u$,$L_{Du}>0$, and $\gamma>\sfrac{1}{2}-H$ such that the inequalities
	\begin{equation}
	\label{eq:ass_u_1}
		\|u_t-u_s\|_{L^q(\Omega)} \leq L_u |t-s|^\gamma
	\end{equation}
and 
	\begin{equation}
	\label{eq:ass_u_2}
		\|D^Hu_t-D^Hu_s\|_{L^q(\Omega;\mathscr{H})} \leq L_{Du}|t-s|^\gamma
	\end{equation}
are satisfied for every $s,t\in [0,1]$. Assume, moreover, that there exist constants $0\leq \alpha< 2H$ and $L>0$ such that the inequality
	\begin{equation}
	\label{eq:ass_u_3}
		\sup_{s\in [0,1]} \|D_s^Hu_t\|_{L^{\sfrac{1}{H}}(\Omega)} \leq Lt^{-\alpha}
	\end{equation}
is satisfied for every $t\in (0,1]$ and denote 
	\begin{equation*}
		i_t(u):=\int_0^tu_s\delta W^H_s := \delta (\bm{1}_{[0,t]}u), \quad t\in [0,1]. 
	\end{equation*}
Then it follows by slightly modifying the proof of \cite[Theorem 4.1]{EssNua15} that for every $t\in [0,1]$, the convergence
	\begin{equation}
	\label{eq:convergence_of_integral}
		\sum_{\substack{j=0,1, \ldots, n \\ \sfrac{i}{n}\leq t}} \left|\int_0^{\frac{(i+1)}{n}} u_s\delta W^H_s - \int_0^{\frac{i}{n}}u_s\delta W^H_s\right|^\frac{1}{H} \quad \underset{n\rightarrow\infty}{\longrightarrow}\quad c_{\frac{1}{H}}\int_0^t |u_s|^\frac{1}{H}\d{s},
	\end{equation}
where $c_{\frac{1}{H}}:=\E|W_1^H|^\frac{1}{H}$, holds in the space $L^1(\Omega)$. Consequently, it follows by \cite[Lemma 1.3]{ConPer18} that $\mathbb{P}$-almost every sample path of the integral process $(i_t(u), t\in [0,1])$ has a finite $\sfrac{1}{H}$-th variation along a subsequence $\tilde{E}=\{E_{n_k}\}_{k\in\mathbb{N}_0}$, where $\{n_k\}_{k\in\mathbb{N}_0}\subseteq\mathbb{N}_0$, $n_k\uparrow\infty$ as $k\uparrow\infty$, of the sequence of equidistant partitions $E=\{E_n\}_{n\in\mathbb{N}}$, where $E_n:=\{\sfrac{k}{n}\,|k=0,1,\ldots,n\}$, of the interval $[0,1]$. 

We finish the example with two remarks. First, we remark that the convergence \eqref{eq:convergence_of_integral} holds even under weaker conditions on the process $u$ than conditions \eqref{eq:ass_u_1} - \eqref{eq:ass_u_3} but we do not state them here in full generality for simplicity of the exposition. The interested reader is referred to \cite[Theorem 4.1]{EssNua15}. Second, note that the process $u\equiv 1$ satisfies conditions \eqref{eq:ass_u_1} - \eqref{eq:ass_u_3} and thus it follows, as a special case of the convergence \eqref{eq:convergence_of_integral}, that the fractional Brownian motion $W^H$ has a finite $\sfrac{1}{H}$-th variation along the sequence $\tilde{E}$. The $\sfrac{1}{H}$-th variation of $W^H$ along the sequence $E$ is also investigated in \cite{Prat11, Tud13}. Additionally, it can also be shown that $W^H$ has a finite $\sfrac{1}{H}$-th variation along a subsequence $\tilde{D}$ of the sequence $D=\{D_n\}_{n\in\mathbb{N}_0}$ of dyadic partitions of the interval $[0,1]$. This follows by using the result of \cite[section 2]{Rog97} together with self-similarity of the fractional Brownian motion to obtain the convergence 
	\begin{equation*}
		\sum_{\substack{k=0,1,\ldots,2^n-1\\\sfrac{k}{2^n}\leq t}}\left|W_{\frac{k+1}{2^n}}^H-W_{\frac{k}{2^n}}^H\right|^\frac{1}{H}\quad \underset{n\rightarrow\infty}{\longrightarrow}\quad c_{\frac{1}{H}} t
	\end{equation*}
in $L^1(\Omega)$ for every dyadic rational $t=\sfrac{L}{2^N}$, $N\in\mathbb{N}$ and $L\in \{0,1,\ldots, 2^N\}$. By a density argument, this convergence is subsequently extended to every $t\in [0,1]$.
\end{example}

\begin{notation}
By the symbol $\mathscr{C}^{1,p}([a,b]\times\R;V)$ where $-\infty<a<b<\infty$ and $p\in\mathbb{N}$ and $V$ is a Hilbert space, we mean the set of functions $f:[a,b]\times \R\rightarrow V$ such that the derivative of $f$ in the first variable $\partial_1f$ exists on $(a,b)\times \R$ and has a continuous extension to $[a,b]\times\R$; and, moreover, the $p$-th partial derivative of $f$ in the second variable $\partial_2^pf$ is continuous and exists on $[a,b]\times\R$.
\end{notation}

The following change-of-variable formula will be needed in the sequel. It is proved in a similar manner as \cite[Theorem 1.5]{ConPer18}.

\begin{lemma}
\label{lem:change_of_variable}
Let $-\infty< a<b<\infty$ and $p\in\mathbb{N}$ be even and let $(V,\langle\cdot,\cdot\rangle_V, \|\cdot\|_V)$ be a Hilbert space. Let $\omega$ be a path such that $\omega\in V_p(\pi)$ for a sequence of partitions $\pi=\{\pi_n\}_{n\in\mathbb{N}}\subset\mathcal{P}[a,b]$ whose mesh size \[\|\pi_n\|:=\max_{[t_j,t_{j+1}]\in\pi_n}|t_{j+1}-t_j|\] tends to zero as $n\rightarrow\infty$. Assume further that $f$ is a function that belongs to the space $\mathscr{C}^{1,p}([a,b]\times\R;V)$. Then the formula
	\begin{equation}
	\label{eq:ito}
		f(t,\omega_t) - f(a,\omega_a)  = \int_a^t(\partial_1f)(u,\omega_u)\d{u} +\frac{1}{p!}\int_a^t(\partial_2^pf)(u,\omega_u)\d[\omega]^\pi_p(u)+ \int_a^t (\partial_2f)(u, \omega_u)\d^\pi\omega_u
	\end{equation}
is satisfied for every $t\in (a,b]$ with the last integral defined as the limit of compensated Riemann sums
	\begin{equation}
	\label{eq:def_int}
		\int_a^t (\partial_2f)(u, \omega_u)\d^\pi\omega_u := \lim_{n\rightarrow\infty}\sum_{[t_j,t_{j+1}]\in\pi_n} \sum_{k=1}^{p-1}\frac{1}{k!} (\partial_2^kf)(t_j,\omega_{t_j}) (\omega_{t_{j+1}\wedge t}-\omega_{t_j\wedge t})^k.
	\end{equation}
\end{lemma}

\begin{proof}
Write
	\begin{align*}
		f(t,\omega_t)-f(a,\omega_a) & = \sum_{[t_j,t_{j+1}]\in\pi_n} \left[f(t_{j+1}\wedge t, \omega_{t_{j+1}\wedge t}) - f(t_{j}\wedge t,\omega_{t_j\wedge t})\right]\\
			& = \sum_{[t_j,t_{j+1}]\in\pi_n} \left[f(t_{j+1}\wedge t, \omega_{t_{j+1}\wedge t}) - f(t_{j}\wedge t,\omega_{t_{j+1}\wedge t})\right] \\
			& \hspace{2cm} +\sum_{[t_j,t_{j+1}]\in\pi_n}\left[f(t_{j}\wedge t,\omega_{t_{j+1}\wedge t})- f(t_{j}\wedge t,\omega_{t_j\wedge t})\right].\numberthis\label{eq:split}
	\end{align*}
In order to treat the first sum notice that there is the estimate
	\begin{align*}
		\sum_{[t_j,t_{j+1}]\in\pi_n}\left\|f(t_{j+1}\wedge t, \omega_{t_{j+1}\wedge t}) - f(t_{j}\wedge t,\omega_{t_{j+1}\wedge t}) - (\partial_1f)(t_j\wedge t,\omega_{t_{j+1}\wedge t})(t_{j+1}\wedge t-t_j\wedge t)\right\|_V & \leq \\
		& \hspace{-11.5cm} \leq \sum_{[t_j,t_{j+1}]\in\pi_n} \left|t_{j+1}\wedge t-t_j\wedge t\right| \sup_{r\in [t_j,t_{j+1}]} \left\|(\partial_1f)(r,\omega_{t_{j+1}\wedge t}) - (\partial_1f)(t_j\wedge t,\omega_{t_{j+1}\wedge t})\right\|_V
	\end{align*}
by \cite[(8.6.2) on p. 162]{Die69}. The right-hand side of the above expression tends to zero as $n\rightarrow\infty$ since $\|\pi_n\|\rightarrow 0$ as $n\rightarrow\infty$ and the convergence of the first sum on the right-hand side of equation \eqref{eq:split} to the first integral on the right-hand side of \eqref{eq:ito} follows. The second sum in \eqref{eq:split} is treated in the same manner as in \cite[Theorem 1.5]{ConPer18}.
\end{proof}

\autoref{lem:change_of_variable} motivates the following definition of admissible integrands for a process with finite $p$-th variation along a sequence of partitions.

\begin{definition}
\label{def:admissible_integrands}
Let $-\infty<a<b<\infty$, $p\in\mathbb{N}$ be even and let $(V,\langle\cdot,\cdot\rangle_V,\|\cdot\|_V)$ be a Hilbert space. Assume that $\omega$ is a function such that $\omega\in V_p(\pi)$ for a sequence of partitions $\pi=\{\pi_n\}_{n\in\mathbb{N}}\subset\mathcal{P}[a,b]$. We say that a function $h:[a,b]\rightarrow V$ is \textit{integrable with respect to} $\omega$ \textit{on the interval $[a,b]$} if there exists a function $f\in\mathscr{C}^{1,p}([a,b]\times\R;V)$ such that $h_t = (\partial_2 f)(t,\omega_t)$ for every $t\in [a,b]$.
\end{definition}

\section{Bilinear evolution equations}
\label{sec:bilinear_equation}

Throughout this section, $(V, \langle\,\cdot\,,\,\cdot\,\rangle_V, \|\,\cdot\,\|_{V})$ is a Hilbert space, and $A:\Dom A\subseteq V\rightarrow V$ and $B:\Dom B\subseteq V\rightarrow V$ be two (not necessarily bounded) linear operators. Moreover, $\omega$ is a function such that $\omega\in V_p(\pi)$ for an even positive integer $p$ and a sequence of partitions $\pi=\{\pi_n\}_{n\in\mathbb{N}}\subset\mathcal{P}[0,T]$, $T>0$, whose mesh size tends to zero. The bilinear problem
	\begin{equation}
		\label{eq:BCP}
		\tag{BLP}
			\begin{cases}
				\d{X}_t & = AX_t\d{t} + BX_t\d^\pi\omega_t, \quad  s< t\leq T, \\
				\phantom{\d}X_s & = x_0
			\end{cases}
	\end{equation}
for $s\in [0,T)$ and $x_0\in V$ is considered in this section.

\begin{definition}
A function $X: [s,T]\rightarrow V$ is said to be a \textit{strong solution} to problem \eqref{eq:BCP} if $X$ takes values in the set $\Dom A \cap \Dom B$, the function $r\mapsto AX_r$ is integrable on the interval $[s,T]$ with respect to the Lebesgue measure, the function $r\mapsto BX_r$ is integrable with respect to $\omega$ on the interval $[s,T]$ in the sense of \autoref{def:admissible_integrands}, and if the equation 
			\begin{equation*}
				X_t = x_0 + \int_s^t AX_r\d{r} + \int_s^t BX_r\d^\pi\omega_r
			\end{equation*}
is satisfied for every $t\in (s,T]$. The second integral in the above equation is defined as the limit of compensated Riemann sums given by formula \eqref{eq:def_int}. 
\end{definition}

\begin{definition}
For every $t\in [0,T]$, let $K(t): \Dom K(t) \subseteq V\rightarrow V$ be a linear operator and consider the homogeneous initial value problem 
	\begin{equation}
	\label{eq:general_evolution_equation}
		\begin{cases}
			\dot{v}(t) & = K(t)v(t), \quad  r<t\leq T,\\
					v(r) & = v_0\\
		\end{cases}
	\end{equation}
for $r\in [0,T)$ and $v_0\in V$. A function $v: [r,T]\rightarrow V$ is called a \textit{solution} to problem \eqref{eq:general_evolution_equation} if $v\in\mathscr{C}^1([r,T];V)$, $v(t)\in\Dom K(t)$ for every $t\in [r,T]$, and $v$ satisfies both equations in \eqref{eq:general_evolution_equation}. A solution $v$ to problem \eqref{eq:general_evolution_equation} is called \textit{$\tilde{V}$-valued} if there exists a set $\tilde{V}$ contained in each $\Dom K(t)$ such that $v_{r,v_0}(t)\in\tilde{V}$ for every $t\in [r,T]$.
\end{definition}

Consider the following assumption.

\begin{assumption} 
\label{ass:1}
\leavevmode
	\begin{itemize}
	\itemsep0em
	\item The $p$-th variation $[\omega]^\pi_p$ of the path $\omega$ is continuously differentiable on the interval $[0,T]$ and we denote its derivative by $\dot{[\omega]}^\pi_p$.
	\item The operator $B$ is an infinitesimal generator of a strongly continuous group of bounded linear operators acting on the space $V$ that is denoted by $G_B$.
	\end{itemize}
\end{assumption}

If \autoref{ass:1} is satisfied by $\omega$ and $B$, set
	\begin{equation}
	\label{eq:dom_C}
		\Dom C:= \Dom A\cap \Dom B^p
	\end{equation} 
and consider the family of linear operators $C(t):\Dom C\rightarrow V$ that are defined by 
	\begin{equation}
	\label{eq:C}
		C(t):=A-\frac{1}{p!}\dot{[\omega]}^\pi_p(t)B^p.
	\end{equation}

\begin{remark}
\label{ex:integrands_have_variation}
Note that the assumption of continuous differentiability of the function $[\omega]^\pi_p$ is not too restrictive. \autoref{ex:stoch_int_fBm} already provides many examples that satisfy this condition. In particular, let $p$ be an even positive integer and let $u:[0,1]\rightarrow\R$ be a deterministic function that belongs to the H\"older space $\mathscr{C}^\beta([0,1])$ with some $\beta>\sfrac{1}{2}-\sfrac{1}{p}$. Then $u$ clearly satisfies conditions \eqref{eq:ass_u_1} - \eqref{eq:ass_u_3}, so that we have, for every $t\in [0,1]$ $\mathbb{P}$-almost surely, 
	\begin{equation*}
		\left[\int_0^\cdot u_s\delta W_s^{\frac{1}{p}}\right]_p^{\tilde{E}} (t) = c_{p}\int_0^t |u_s|^p\d{s}
	\end{equation*}
which is continuously differentiable in $t$ since $u$ is continuous. 
\end{remark}

\subsection{The commutative case}
\label{sec:comm}

In this section, a strong solution to \eqref{eq:BCP} is found under the assumption that the operator $A$ commutes with the group $G_B$ (on an appropriate domain). Consider the homogeneous initial value problem 
	\begin{equation}
	\label{eq:CP}
	\tag{CP}
		\begin{cases}
			\dot{v}(t) & = C(t)v(t), \quad  s< t\leq T,\\
			v(s) & = x_0.
		\end{cases}
	\end{equation}
The next proposition connects the solution to problem \eqref{eq:CP} to the solution to problem \eqref{eq:BCP}.

\begin{proposition}
\label{prop:linear_equation}
Let \autoref{ass:1} be verified. Assume furthermore that there exists a $\tilde{V}$-valued solution $v_{s,x_0}$ to the problem \eqref{eq:CP} with $\tilde{V}$ being a subset of $\Dom C$ that is closed under the action of the group $G_B$, i.e. $G_B\tilde{V}\subseteq \tilde{V}$, and such that the following condition is satisfied:
	\begin{enumerate}[label={\textnormal{(AB)}}, leftmargin=\widthof{(AB)}+\labelsep]
	\itemsep0em
	\item\label{ass:AB} The operator $A$ commutes with the group $G_B$ on $\tilde{V}$, i.e. the equation 
		\begin{equation*}
		AG_B(t)y = G_B(t)Ay
	\end{equation*}
is satisfied for every $y\in \tilde{V}$ and $t\in\R$. 
	\end{enumerate}
Then the function $X_{s,x_0}^\omega: [s,T]\rightarrow V$ defined by 
	\begin{equation*}
		X_{s,x_0}^\omega(t) := G_B(\omega_t-\omega_s)v_{s,x_0}(t)
	\end{equation*}
is a strong solution to the problem \eqref{eq:BCP}.
\end{proposition}

\begin{proof}
Define $f: [s,T]\times \R\rightarrow V$ by 
	\begin{equation*}
		f(t,x) := G_B(x-\omega_s)v_{s,x_0}(t).
	\end{equation*}
For every $v\in\Dom B^p$, the map $x\mapsto G_B(x-\omega_s)v$ belongs to $\mathscr{C}^p(\R;V)$. Moreover, we have that $v_{s,x_0}$ belongs to $\mathscr{C}^1([s,T];V)$ and $v_{s,x_0}(t)\in \tilde{V}\subseteq\Dom C\subseteq \Dom B^p$ for every $t\in [s,T]$. It follows that $f\in\mathscr{C}^{1,p}([s,T]\times\R;V)$ and its partial derivatives are given by
	\begin{equation*}
		(\partial_1f)(t,x) = G_B(x-\omega_s)\dot{v}_{s,x_0}(t)
	\end{equation*}
and
	\begin{equation*}
		(\partial_2^kf)(t,x) = B^kG_B(x-\omega_s)v_{s,x_0}(t)
	\end{equation*}
for $(t,x)\in [s,T]\times\R$ and $k=1,2,\ldots, p$. Therefore, by \autoref{lem:change_of_variable} we have that 
	\begin{align*}
		G_B(\omega_t-\omega_s)v_{s,x_0}(t) & = x_0 + \int_s^t BG_B(\omega_r-\omega_s)v_{s,x_0}(r)\d^\pi\omega_r \\
		& \hspace{2cm} + \frac{1}{p!} \int_s^t B^pG_B(\omega_r-\omega_s)v_{s,x_0}(r)\d[\omega]_p^\pi(r)\\
		& \hspace{2cm} + \int_s^t G_B(\omega_r-\omega_s)\dot{v}_{s,x_0}(r)\d{r}
	\end{align*}
By using the fact that $v_{s,x_0}$ satisfies equation \eqref{eq:CP}, we have that
	\begin{align*}
		\int_s^t G_B(\omega_r-\omega_s)\dot{v}_{s,x_0}(r)\d{r} & = \int_s^t G_B(\omega_r-\omega_s)C(r)v_{s,x_0}(r)\d{r} \\
		& = \int_s^t G_B(\omega_r-\omega_s)\left(A-\frac{1}{p!}B^p\dot{[\omega]}^\pi_p(r)\right)v_{s,x_0}(r)\d{r}\\
		& = \int_s^tAG_B(\omega_r-\omega_s)v_{s,x_0}(r)\d{r} - \frac{1}{p!} \int_s^t B^pG_B(\omega_r-\omega_s)v_{s,x_0}(r)\d[\omega]^\pi_p(r)
	\end{align*}
where we successively used the assumptions of commutativity of $A$ and $G_B$ on the set $\tilde{V}$, continuous differentiability of $[\omega]^\pi_p$; and the fact that the operator $B^p$ commutes with the group $G_B$ on the set $\Dom B^p$, see \cite[Theorem 2.4]{Pazy}. The claim follows.
\end{proof}

\begin{example}
\label{ex:one-dim_eq}
Let $W^H=(W_t^H, t\in [0,1])$ be the fractional Brownian motion with Hurst parameter $H\in (0,1)$ defined on a probability space $(\Omega,\mathcal{F},\mathbb{P})$. By \autoref{ex:stoch_int_fBm} we have that $\mathbb{P}$-almost every sample path of $W^H$ has a finite $\sfrac{1}{H}$-th variation along the sequence $\tilde{D}$ (recall that $\tilde{D}$ is a subsequence of the sequence $D$ of dyadic partitions of the interval $[0,1]$). By \autoref{prop:linear_equation}, we obtain that for every $k\in\mathbb{N}$, the process $(Y_{k,1}(t),t\in [0,1])$ defined by the formula
	\begin{equation*}
		Y_{k,1}(t) := \mathrm{exp}\,\left\{bW^{\frac{1}{2k}}_t + \left(a-\frac{c_{2k}}{(2k)!}b^{2k}\right)t\right\},
	\end{equation*}
where $c_{2k}=\E|W^\frac{1}{2k}_1|^{2k}$ as before, satisfies the equation 
	\begin{equation*}
		Y_{k,1}(t) = 1 + \int_0^t aY_{k,1}(r)\d{r} + \int_0^t bY_{k,1}(r)\d^{\tilde{D}}\hspace{-0.5mm}W^\frac{1}{2k}_r
	\end{equation*}
for every $t\in [0,1]$ $\mathbb{P}$-almost surely (the second index is used to distinguish between the solutions that are obtained in this and in the following two examples). The formula for the process $Y_{1,1}$ is well-known from classical (It\^o's) stochastic calculus. In fact, if $k=1$, we have that $W^\frac{1}{2}$ is the Wiener process (that we denote by $W$ everywhere in the paper) and the process satisfies the equation 
	\begin{equation*}
		Y_{1,1}(t) = 1 + \int_0^t aY_{1,1}(r)\d{r} + \int_0^t bY_{1,1}(r)\d W_r
	\end{equation*}
for every $t\in [0,1]$ $\mathbb{P}$-almost surely where the integral $\int_0^t (...)\d W_r$ is the usual It\^o integral. On the other hand, already in the case $k=2$, the form of the solution differs from similar results that can be found in the literature. In particular, by \cite[Theorem 3.4]{Snup10} (see also \cite[Example 3.1]{Snup12}) it follows that the process $(\tilde{Y}_{2,1}(t), t\in [0,1])$ defined by 
	\begin{equation*}
		\tilde{Y}_{2,1}(t) := \mathrm{exp}\,\left\{b W_t^\frac{1}{4} + at - \frac{1}{2}b^2\sqrt{t}\right\}
	\end{equation*}
satisfies the equation 
	\begin{equation*}
		\tilde{Y}_{2,1}(t) = 1 + \int_0^ta \tilde{Y}_{2,1}(r)\d{r} + \int_0^t b \tilde{Y}_{2,1}(r)\bar{\delta} W_r^\frac{1}{4}
	\end{equation*}
for every $t\in [0,1]$ $\mathbb{P}$-almost surely where the integral $\int_0^t (...)\bar{\delta}W_r^\frac{1}{4}$ is the extension of the Skorokhod integral with respect to the fractional Brownian motion from \autoref{ex:stoch_int_fBm}, that is introduced in \cite{CheNua05}.
\end{example}

\begin{example}
\label{ex:inf-dim_eq_bdd_B}
More generally, consider the problem 
	\begin{equation}
	\label{eq:equation}
		\begin{cases}
			\d X_t & = AX_t\d{t} + bX_t\d^{\tilde{D}}W^\frac{1}{2k}, \quad 0<t\leq 1,\\
			\phantom{\d}X_0 & = x_0,
		\end{cases}
	\end{equation}
where $A: \Dom A\subseteq V\rightarrow V$ is an infinitesimal generator of a strongly continuous semigroup $S_A$ of bounded linear operators acting on the space $V$ and $b\in\mathbb{R}\setminus\{0\}$. 

\medskip

As in the above \autoref{ex:one-dim_eq}, the $\sfrac{1}{2k}$-th variation of almost every sample path of the fractional Brownian motion $W^{\sfrac{1}{2k}}$ along $\tilde{D}$ is continuously differentiable. The operator $B:=b\,\mathrm{Id}_V$ generates a strongly continuous group $G_{b\mathrm{Id}_V}$ given by $G_{b\mathrm{Id}_V}(t)=\mathrm{e}^{bt}\mathrm{Id}_V$ for $t\in\R$ so that \autoref{ass:1} is satisfied. The space $\Dom C$ that is defined by \eqref{eq:dom_C} equals $\Dom A$ and we will show that the remaining assumptions of \autoref{prop:linear_equation} are satisfied with $\tilde{V}=\Dom A$. We have that $\Dom A$ is closed under the action of the group $G_{b\mathrm{Id}_V}$ since the action of this group is only a multiplication by $\mathrm{e}^{bt}$. By the same argument, it follows that condition \ref{ass:AB} is satisfied as well. The system of operators $C(t)$ that is defined by \eqref{eq:C} is independent of $t$ and it is given by $C(t)=C$ where $C$ is given by \[C:=A-\frac{c_{2k}}{(2k)!}b^{2k}\mathrm{Id}_V\] on $\Dom A$. The operator $C$ generates a strongly continuous semigroup $S_C$ by \cite[Theorem 1.1 in chapter 3]{Pazy} that is given by \[S_C(t) = \mathrm{exp}\,\left\{-\frac{c_{2k}}{(2k)!}b^{2k}\,t\right\}S_A(t)\] for $t\geq 0$ by \cite[formula (1.2) on p. 77]{Pazy}. This semigroup maps $\Dom A$ back into itself by \cite[Theorem 2.4, c), in chapter 1]{Pazy} and so that for every $x_0\in\Dom A$, there is a $\Dom A$-valued solution to the problem
	\begin{equation*}
		\begin{cases}
		\dot{v}(t) & = Cv(t), \quad 0< t\leq 1,\\
		v(0) & = x_0,
		\end{cases}
	\end{equation*}
that is given by $v_{0,x_0}(t)=S_C(t)x_0$. Consequently, for every $x_0\in\Dom A$, it follows by \autoref{prop:linear_equation} that the process  
	\begin{equation*}
		Y_{k,2}(t) := \mathrm{exp}\,\left\{bW^\frac{1}{2k}_t - \frac{c_{2k}}{(2k)!}b^{2k}\,t\right\}S_A(t)x_0, \quad 0\leq t\leq 1,
	\end{equation*}
satisfies the equation
	\begin{equation*}
		Y_{k,2}(t) = x_0 + \int_0^t AY_{k,2}(r)\d{r} + \int_0^tbY_{k,2}(r)\d^{\tilde{D}}W^{\frac{1}{2k}}_r
	\end{equation*}
for every $t\in [0,1]$ $\mathbb{P}$-almost surely. An example to which this result can be applied is the initial-boundary value problem for the heat equation that is formally described by 
	\begin{equation*}
		(\partial_tu)(t,x) = (\Delta_x u)(t,x) + bu(t,x)\dot{W}_t^\frac{1}{2k}
	\end{equation*}
for $(t,x)\in [0,1]\times\mathcal{O}$ where $\mathcal{O}\subset\R^d$, $d\in\mathbb{N}$, is a bounded domain with smooth boundary $\partial\mathcal{O}$ and $\Delta_x$ is the Lapace operator. The equation is subject to the initial condition $u(0,x)=u_0$ and the Dirichlet boundary condition
	\begin{equation*}
		u(t,x)=0
	\end{equation*}
or the Neumann boundary condition 
	\begin{equation*}
		\frac{\partial u}{\partial\nu}(t,x) = 0
	\end{equation*}
for $(t,x)\in [0,1]\times\partial\mathcal{O}$. In the Neumann problem, the symbol $\frac{\partial}{\partial\nu}$ denotes the conormal derivative. The above problem is rigorously interpreted as equation \eqref{eq:equation} by setting $V:=L^2(\mathcal{O})$, $A:=\Delta_x$ on $\Dom A:=W^{2,2}(\mathcal{O})\cap W^{1,2}_0(\mathcal{O})$ for the Dirichlet problem and on $\Dom A:=\{f\in W^{2,2}(\mathcal{O})\,\left|\right.\, \frac{\partial f}{\partial\bm{\nu}}=0\mbox{ on } \partial\mathcal{O}\}$ for the Neumann problem.
\end{example}

\begin{example}
\label{ex:B_derivative}
Consider the formal equation 
	\begin{equation*}
		(\partial_tu)(t,x) = a(\partial_x^2u)(t,x) + b(\partial_xu)(t,x)\dot{W}_t^{\sfrac{1}{2k}}
	\end{equation*}
for $(t,x)\in [0,1]\times\R$ with the initial condition $u(0,x)=u_0(x)$ for $x\in\R$ where $a\in\R$ and $b\in\mathbb{R}\setminus\{0\}$. Here again, the process $\dot{W}^{\sfrac{1}{2k}}$ is the formal time derivative of the fractional Brownian motion with the Hurst parameter $H=\sfrac{1}{2k}$ for some $k\in\mathbb{N}$ that is defined on a probability space $(\Omega,\mathcal{F},\mathbb{P})$. 

\medskip

Rigorously, the above problem can be written in our framework by setting $\omega:=W^{\sfrac{1}{2k}}$, $V:=L^2(\R)$, $A:= a\partial^{2}$ on $\Dom A:=W^{2,2}(\R)$ (where $[(a\partial^2)f](x):=af''(x)$ for $f\in W^{2,2}(\R)$), $B:= b\partial$ on $\Dom B:=W^{1,2}(\R)$ (where $[(b\partial)f](x):=bf'(x)$ for $f\in W^{1,2}(\R)$). Below, it is shown that this equation can be solved by our method if $k=1$ and $a>\sfrac{b^2}{2}$ or if $k$ is an even positive integer, i.e. if the driving signal is either the Wiener process or a fractional Brownian motion with Hurst parameter $H=\sfrac{1}{4m}$ for $m\in\mathbb{N}$. 

\medskip

As in \autoref{ex:stoch_int_fBm}, we have that $[W^{\sfrac{1}{2k}}]_{2k}^{\tilde{D}}(t)=c_{2k}t$ which is a continuously differentiable function. The operator $b\partial$ generates a strongly continuous group $G_{b\partial}$ on the space $L^2(\R)$ that is given by 
	\begin{equation*}
		[G_{b\partial}(t)f](x) = f(x+bt)
	\end{equation*}
for $t,x\in\R$ and $f\in L^2(\R)$, see e.g. \cite[Proposition 1 on p. 66]{EngNag00}, and hence, \autoref{ass:1} is satisfied. The operators $C(t)$ defined by \eqref{eq:C} are independent of $t$ and they are given by $C(t)=C_k$ where
	\begin{equation}
	\label{eq:C_ex}
		C_k=a\partial^2- \frac{c_{2k}}{(2k)!}b^{2k}\partial^{2k}
	\end{equation}
on the space $\Dom C= W^{2k,2}(\R)$ and since $G_{b\partial}(t)$ is simply a shift operator for every $t\in\R$, it maps $W^{2k,2}(\R)$ back into itself. We may therefore set $\tilde{V}:=W^{2k,2}(\R)$. We have that the equality 
	\begin{equation*}
		[G_{b\partial}(t)(a\partial^2)f](x) = a f''(x+bt) = [(a\partial^2)G_{b\partial}(t)f](x)
	\end{equation*}
holds for every $f\in W^{2k,2}(\R)$ and every $t,x\in\R$ so that condition \ref{ass:AB} is satisfied as well. It remains to show that the corresponding homogeneous initial value problem \eqref{eq:CP} has a $\tilde{V}$-valued solution. Below, we split the reasoning into two cases that will be treated separately.

\medskip

\textit{The Wiener case ($k=1$).} If $k=1$, we have that the operator $C_1$ is given by  $C_1=c\partial^2$ with $c:=a-\frac{b^2}{2}$ and it is defined on the domain $\Dom C_1 = W^{2,2}(\R)$. If $c>0$, then the operator $-c\partial^2$ is strongly elliptic, see \cite[Definition 2.1]{Pazy}, and hence, by \cite[Theorem 2.7 and Remark 2.9 in chapter 7]{Pazy} it follows that the operator $c\partial^2$ generates an analytic semigroup $S_{c\partial^2}$ of bounded linear operators on the space $L^2(\R)$. This is of course the scaled heat semigroup given by 
	\begin{equation*}
		[S_{c\partial^2}(t)f](x) = \frac{1}{\sqrt{4\pi ct}}\int_{\R} \mathrm{e}^{-\frac{(x-y)^2}{4ct}}f(y)\d{y}
	\end{equation*}
for $t> 0$, $x\in\R$, and $f\in L^2(\R)$. Hence, we have that for every $f\in W^{2,2}(\R)$, there is a $W^{2,2}(\mathbb{R})$-valued solution to the problem
	\begin{equation*}
		\begin{cases}
			\dot{v}(t) & = C_1 v(t), \quad 0<t\leq 1,\\
			v(0) & = f
		\end{cases}
	\end{equation*}
that is given by $v_{0,f}(t) = S_{c\partial^2}(t)f$. Consequently, by \autoref{prop:linear_equation}, we obtain that the process
	\begin{equation*}
		Y_{1,3}(t) := G_{b\partial}(W_t)S_{c\partial^2}(t)f
	\end{equation*}
satisfies the equation 
	\begin{equation*}
		Y_{1,3}(t) = f + \int_0^t a\partial^2 Y_{1,3}(r)\d{r} + \int_0^tb\partial Y_{1,3}(r)\d^{\tilde{D}}W_r
	\end{equation*}
for every $t\in [0,1]$ $\mathbb{P}$-almost surely. Moreover, in this case, there is the explicit formula 
	\begin{equation*}
		Y_{1,3}(t,x) = \frac{1}{\sqrt{4\pi ct}}\int_{\R}\mathrm{e}^{-\frac{(x-y)^2}{4ct}}f(y+bW_t)\d{y}, \quad t\in (0,1], x\in\R.
	\end{equation*}

\textit{The singular fractional Brownian motion case $(k>1)$.} In this case, the operator $-C_k$ that is given by \eqref{eq:C_ex} is strongly elliptic if $k$ is an even integer. For $k=2m$, $m\in\mathbb{N}$, it follows by \cite[Theorem 2.7 and Remark 2.9 in chapter 7]{Pazy} that the operator $C_{2m}$ generates an analytic semigroup $S_{C_{2m}}$ of bounded linear operators acting on the space $L^2(\R)$. Therefore, as in the Wiener case above, we have by \autoref{prop:linear_equation} that for every $f\in W^{4m,2}(\R)$, the $W^{4m,2}(\R)$-valued process 
	\begin{equation*}
		Y_{2m,3}(t) := G_{b\partial}(W^\frac{1}{4m}_t) S_{C_k}(t)f
	\end{equation*}
satisfies the equation
	\begin{equation*}
		Y_{2m,3}(t) = f + \int_0^t a\partial^2 Y_{2m,3}(r)\d{r} + \int_0^t b\partial Y_{2m,3}(r)\d^{\tilde{D}}W^{\frac{1}{4m}}_r
	\end{equation*}
for every $t\in [0,1]$ $\mathbb{P}$-almost surely.
\end{example}

\subsection{The non-commutative case}
\label{sec:non-comm}

In what follows, we show that the commutativity assumption \ref{ass:AB} in \autoref{prop:linear_equation} can be weakend in the spirit of \cite{DaPIanTub82} (see also \cite[section 6.5]{DaPratoZab14}). Set $z_{s,\omega}(t):=G_B(\omega_t-\omega_s)$ for $t\in[s,T]$ and note that for every $t\in [s,T]$, the operator $z_{s,\omega}(t)$ is invertible with $z_{s,\omega}^{-1}(t) = G_B(\omega_s-\omega_t)$. Consider the homogeneous initial value problem   
	\begin{equation}
	\label{eq:NCP}
	\tag{NCP}
	\begin{cases}
		\dot{v}(t) & = z_{s,\omega}^{-1}(t)C(t)z_{s,\omega}(t)v(t), \quad 0\leq s<t\leq T, \\
		v(s)& = x_0.
	\end{cases}
	\end{equation}
Similarly as in \autoref{prop:linear_equation} it is now shown that if problem \eqref{eq:NCP} has a solution, then the bilinear problem \eqref{eq:BCP} also has a solution; however, without assuming that $A$ and $G_B$ commute. More precisely, we have the following result:

\begin{proposition}
\label{prop:linear_equation_without_commutativity}
Let \autoref{ass:1} be verified. Assume that there exists a $\tilde{V}$-valued solution $v_{s,x_0}$ to problem \eqref{eq:NCP} with $\tilde{V}$ being a subset of $\Dom C$ that is closed under the action of the group $G_B$, i.e. $G_B\tilde{V}\subseteq \tilde{V}$. Then the function $X_{s,x_0}^\omega:[s,T]\rightarrow V$ defined by 
	\begin{equation*}
		X_{s,x_0}^\omega(t) := z_{s,\omega}(t)v_{s,x_0}(t)
	\end{equation*}
is a strong solution to the problem \eqref{eq:BCP}.
\end{proposition}

\begin{remark}
Note that \autoref{prop:linear_equation_without_commutativity} is a generalization of \autoref{prop:linear_equation}. Indeed, if condition \ref{ass:AB} is satisfied, then problem \eqref{eq:NCP} becomes problem \eqref{eq:CP}.
\end{remark}

\begin{proof}[Proof of \autoref{prop:linear_equation_without_commutativity}]
Define $f: [s,T]\times\R\rightarrow V$ by 
	\begin{equation*}
		f(t,x):= G_B(x-\omega_s)v_{s,x_0}(t).
	\end{equation*}
As in the proof of \autoref{prop:linear_equation}, we apply \autoref{lem:change_of_variable} to obtain 
	\begin{align*}
		z_{s,\omega}(t)v_{s,x_0}(t) & = x_0 + \int_s^t Bz_{s,\omega}(r)v_{s,x_0}(r)\d^\pi\omega_r\\
		& \hspace{2cm} + \frac{1}{p!} \int_s^t B^p z_{s,\omega}(r)v_{s,x_0}(r)\d[\omega]^\pi_p(r)\\
		& \hspace{2cm} + \int_s^t z_{s,\omega}(r)\dot{v}_{s,x_0}(r)\d{r}.
	\end{align*}
By using the fact that $v_{s,x_0}$ satisfies equation \eqref{eq:NCP}, we have that 
	\begin{align*}
		\int_s^t z_{s,\omega}(r)\dot{v}_{s,x_0}(r)\d{r} & = \int_s^t z_{s,\omega}(r)z_{s,\omega}^{-1}(r)C(r)z_{s,\omega}(r)v_{s,x_0}(r)\d{r}\\
		& = \int_s^t \left(A-\frac{1}{p!}\dot{[\omega]}_p^\pi(r)B^p\right)z_{s,\omega}(r)v_{s,x_0}(r)\d{r}\\
		& = \int_s^t Az_{s,\omega}(r)v_{s,x_0}(r)\d{r} -\frac{1}{p!}\int_s^t B^pz_{s,\omega}(r)v_{s,x_0}(r)\d[\omega]_p^\pi(r)
	\end{align*}
which proves the claim.
\end{proof}

\begin{remark}
In \autoref{prop:linear_equation_without_commutativity}, the solution $X_{s,x_0}^\omega$ belongs to the space $\mathscr{C}([s,T];V)$. 
\end{remark}

In the following two subsections, we give sufficient conditions for existence of a strong solution to the problem \eqref{eq:BCP} by solving the equation \eqref{eq:NCP} in two cases: in the parabolic case and in the hyperbolic case.

\subsubsection{The parabolic non-commutative case}

In what follows we give some sufficient conditions for the solvability of the system  \eqref{eq:NCP} in what is usually called the parabolic case. Recall that for any linear operator $K:\Dom K\subseteq V\rightarrow V$, the symbol $\rho(K)$ denotes the \textit{resolvent set}, i.e. the set of all $\lambda\in\mathbb{C}$ such that the operator $(\lambda\mathrm{Id}_V-K)$ has bounded inverse, and the operator $R(\lambda:K):=(\lambda\mathrm{Id}_V-K)^{-1}$ defined for $\lambda\in\rho(K)$ is the \textit{resolvent operator}.
We say that a family of operators $(K(t),t\in [0,T])$ is \textit{parabolic}\footnote{The term \textit{parabolic} is used since assumption \ref{ass:P2} says that the operators $K(t)$ are sectorial.} if the following three conditions are satisfied:
	\begin{enumerate}[label=(P\arabic*)]
		\itemsep0em
		\item\label{ass:P1} The family $(K(t),t\in [0,T])$ consists of closed linear operators on the space $V$ that are defined on a common domain $D$ which is independent of $t$ and dense in $V$.
		\item\label{ass:P2} There exists $\lambda_K\in\mathbb{R}$ such that, for every $t\in [0,T]$, the resolvent set of the operator $K(t)$ contains the half-plane \[\mathbb{C}^+_{\lambda_K}:=\{\lambda\in\mathbb{C}\,\left|\right. \Re[\lambda]\geq \lambda_K\}\] and there exists a constant $M>0$ such that the inequality
			\begin{equation*}
				\left\|R(\lambda: K(t))\right\|_{\mathscr{L}(V)} \leq \frac{M}{1+ |\lambda-\lambda_K|}
			\end{equation*}
		is satisfied for every $\lambda\in\mathbb{C}^+_{\lambda_K}$ and $t\in [0,T]$. 
		\item\label{ass:P3} There exist constants $L>0$ and $0<\alpha\leq 1$ such that inequality 
			\begin{equation*}
				\left\| [K(t)-K(s)](\lambda_K\mathrm{Id}_V-K(0))^{-1}\right\|_{\mathscr{L}(V)}\leq L|t-s|^\alpha
			\end{equation*}
		is satisfied for every $s,t\in [0,T]$. 
	\end{enumerate}

\begin{proposition}
\label{prop:parabolic_noncom}
Assume that $\omega$ satisfies the following two regularity conditions:
\begin{enumerate}[label=\normalfont (p\arabic*)]
\itemsep0em
	\item\label{ass:omega_1} There exists constants $k_1>0$ and $0<\gamma_1< 1$ such that 
		\begin{equation*}
			|\omega_u-\omega_v|\leq k_1|u-v|^{\gamma_1}
		\end{equation*}
	holds for every $u,v\in [0,T]$.
		\item\label{ass:omega_2} The $p$-th variation of the path $\omega$ is continuously differentiable on the interval $[0,T]$ with the derivative being denoted by $\dot{[\omega]}^\pi_p$. Moreover, there exist constants $k_2>0$ and $0<\gamma_2\leq 1$ such that 
		\begin{equation*}
			\left|\dot{[\omega]}_p^\pi(u)-\dot{[\omega]}_p^\pi(v)\right| \leq k_2 |u-v|^{\gamma_2}
		\end{equation*}	
	is satisfied for $u,v\in [0,T]$.
\end{enumerate}
Assume further, that the operator $B$ satisfies the following two conditions:
\begin{enumerate}[label=\normalfont (p\arabic*)]
\itemsep0em
\setcounter{enumi}{2}
	\item\label{ass:B_1} The operator $B$ is an infinitesimal generator of a strongly continuous group of bounded linear operators $G_B$ acting on the space $V$.
	\item\label{ass:B_2} The operator $B^p$ is closed.
\end{enumerate}
Assume that the space $\Dom C$ defined by formula \eqref{eq:dom_C} and the family $(C(t), t\in [0,T])$ defined by formula \eqref{eq:C} satisfy the following two conditions:
\begin{enumerate}[label=\normalfont (p\arabic*)]
\itemsep0em
\setcounter{enumi}{4}
	\item\label{ass:C_1} The set $\Dom C$ is dense in the space $V$ and for every $t\in [0,T]$, the operator $C(t)$ is closed.
	\item\label{ass:C_2} There exists $\lambda_0\in\mathbb{R}$ such that the resolvent set $\rho(C(t))$ contains the half-plane $\mathbb{C}^+_{\lambda_0}$ for every $t\in [0,T]$ and there exists a constant $k_3>0$ such that the inequality 
		\begin{equation*}
			\|R(\lambda: C(t))\|_{\mathscr{L}(V)} \leq \frac{k_3}{1+|\lambda-\lambda_0|}
		\end{equation*}
	holds for every $\lambda\in\mathbb{C}^+_{\lambda_0}$ and $t\in [0,T]$. 
\end{enumerate}
Finally, assume that $B$ and the family $(C(t), t\in [0,T])$ are connected in the following manner:
\begin{enumerate}[label=\normalfont (p\arabic*)]
\itemsep0em
\setcounter{enumi}{6}
	\item\label{ass:BC} There exist $\theta\in\mathbb{R}$, $K\subseteq\Dom B$, and a family $(L(t), t\in[0,T])$ of bounded linear operators acting on the space $V$ such that $K$ is the core of $B$ and such that for every $t\in [0,T]$, $\theta\in\rho(C(t))$ and the equality
		\begin{equation*}
			 [\theta-C(t)]B[\theta-C(t)]^{-1}z = [B+ L(t)]z
		\end{equation*}
	is satisfied for every $z\in K$. 
\end{enumerate}
For $\alpha\geq 0$ and $r\in [0,T]$, denote by  $D_\alpha(r)$ the domain of the fractional power\footnote{The fractional powers can be defined since assumptions \ref{ass:C_1} and \ref{ass:C_2} imply that the operator $C(r)$ defines an analytic semigroup of bounded linear operators on $V$ and that the operator $\lambda_0\mathrm{Id}_V-C(r)$ is positive, see e.g. \cite[section 2.6]{Pazy}.} $(\lambda_0\mathrm{Id}_V-C(r))^\alpha$ of the operator $\lambda_0\mathrm{Id}_V-C(r)$. Let $s\in [0,T)$ and $x_0\in V$. If there exists $\varepsilon>0$ such that $x_0\in D_{1+\varepsilon}(s)$, then there is a strong solution to the bilinear problem \eqref{eq:BCP}. 
\end{proposition}

\begin{remark}
\autoref{prop:parabolic_noncom} is reminiscent of \cite[Proposition 1]{DaPIanTub82} and the proof technique is taken from there. There are, however, several reasons why it seems necessary to include the full proof here. 
	\begin{itemize}
	\itemsep0em
		\item The first difference is that \autoref{prop:parabolic_noncom} allows for a large number of driving paths $\omega$ while \cite[Proposition 1]{DaPIanTub82} only deals with the case when $\omega$ is a path of the Wiener process. 
		\item The second difference is that in \cite[Proposition 1]{DaPIanTub82}, the operator $C(t)$ does not depend on $t$ since the usual quadratic variation of the Wiener process has constant derivative. As it is seen from the following proof, the dependence of $C(t)$ on $t$ causes some technical difficulties. 
		\item Finally, in \cite[Proposition 1]{DaPIanTub82}, it could be assumed that $\Dom A\subseteq\Dom B^2$ since typically, the operator $A$ exhibits ``worse" behaviour than $B^2$; however, in our case, the situation is in many cases reversed. This is already seen in \autoref{ex:B_derivative} where $B$ is a first-order differential operator so that $B^p$ is $p$-th order differential operator (recall that $p$ is a positive even integer so that $p\geq 2$) and $A$ is a second-order differential operator so that $\Dom B^p\subseteq\Dom A$. Because of this, the roles of $\Dom A$ and $\Dom B^p$ are symmetric in \autoref{prop:parabolic_noncom}.
	\end{itemize}
\end{remark}

\begin{proof}[Proof of \autoref{prop:linear_equation_without_commutativity}]
Without loss of generality, it is assumed that $\lambda_0=0$. Moreover, in order to simplify the exposition, only the case $\theta=0$ shall be proved. We divide the proof into several steps.
\medskip\\
\textit{Step 1:} Let $s\in [0,T)$ be fixed for the rest of the proof. First it is shown that the space $\Dom C$ is closed under the action of the group $G_B$ so that the operator 
	\begin{equation*}
		\tilde{C}_s(t):= z_{s,\omega}^{-1}(t)C(t)z_{s,\omega}(t) = G_B^{-1}(\omega_t-\omega_s)C(t)G_B(\omega_t-\omega_s)
	\end{equation*}
on $\Dom \tilde{C}_s(t) =\Dom C$ is well-defined for every $t\in [0,T]$. Indeed, note that the equality
	\begin{equation}
	\label{eq:factor_resolvent_C}
		C(t)R(\lambda: B)C^{-1}(t) = R(\lambda:B+L(t))
	\end{equation}
holds for every $t\in [0,T]$ and $\lambda\in \rho(B)\cap\rho(B+L(t))$. This is proved in the same manner as \cite[formula (5)]{DaPIanTub82} by using assumption \ref{ass:BC}. By \cite[Theorem 1.1 in section 3.1]{Pazy}, we have that for every $t\in [0,T]$, the operator $B+L(t)$ is the infinitesimal generator of a strongly continuous group of bounded linear operators acting on the space $V$ that we denote by $G_{B+L(t)}$ and that satisfies \[\|G_{B+L(t)}(x)\|_{\mathscr{L}(V)}\leq M\mathrm{exp}\left\{(m+M\|L(t)\|_{\mathscr{L}(V)})|x|\right\}\] for every $x\in \R$ where $M\geq 1$ and $m\geq 0$ are constants such that the inequality \[\|G_B(x)\|_{\mathscr{L}(V)}\leq M\mathrm{exp}\{m|x|\}\] holds for every $x\in\R$. This follows by the fact that the operator $B$ generates a strongly continuous group $G_B$ by assumption \ref{ass:B_1}. It follows from \eqref{eq:factor_resolvent_C} that the space $\Dom C$ is closed under the action of the group $G_B$ and that for every $x\in\R$ and $t\in [0,T]$, the equality
	\begin{equation}
	\label{eq:representation_of_G(B+L)}
		C(t)G_B(x)C^{-1}(t) = G_{B+L(t)}(x)
	\end{equation}
is satisfied. 
\medskip\\
Before we continue, let us fix the following notation that will simplify the exposition. Set \[\Omega_{0,T}:=\{x\in\R\,|\, x=|\omega_u-\omega_v|,\, u,v\in [0,T]\}\] and note that this set is compact since the path $\omega$ is continuous. In particular, the number $\sup\, \Omega_{0,T}$ is finite. Moreover, for fixed $u,v\in [0,T]$, we have
	\begin{equation*}
		\|G_B(\omega_u-\omega_v)\|_{\mathscr{L}(V)} \leq M\mathrm{exp}\,\{m|\omega_u-\omega_v|\} \leq M\mathrm{exp}\,\{m\sup\Omega_{0,T}\} =: k_\omega.
	\end{equation*}
	
\textit{Step 2:} In the second step, it is shown that the family $(\tilde{C}_s(t), t\in [0,T])$ is parabolic. Indeed, it is clear by assumptions \ref{ass:C_1} and \ref{ass:C_2} that the family  $(\tilde{C}_s(t), t\in [0,T])$ satisfies conditions \ref{ass:P1} and \ref{ass:P2}. We will show that the condition \ref{ass:P3} is satisfied as well. To this end, let $u,v\in [0,T]$ be arbitrary. We have that the inequality
	\begin{align*}
		\|\tilde{C}_s(u)\tilde{C}_s^{-1}(v)-\mathrm{Id}_V\|_{\mathscr{L}(V)}\leq k_{\omega}^2 \|C(u)G_B(\omega_u-\omega_v)C^{-1}(v) - G_B(\omega_u-\omega_v)\|_{\mathscr{L}(V)}
	\end{align*}
is satisfied and adding and subtracting the term $C(v)G_B(\omega_u-\omega_v)C^{-1}(v)$ inside the norm yields the inequality
	\begin{equation*}
		\|\tilde{C}_s(u)\tilde{C}_s^{-1}(v)-\mathrm{Id}_V\|_{\mathscr{L}(V)} \leq k_{\omega}^2\left[ \mathrm{(I)} + \mathrm{(II)}\right]
	\end{equation*}
where the expressions $\mathrm{(I)} $ and $\mathrm{(II)}$ are defined by 
	\begin{align*}
		& \mathrm{(I)}  :=  \|[C(u)-C(v)]G_B(\omega_u-\omega_v)C^{-1}(v)\|_{\mathscr{L}(V)}, \\
		& \mathrm{(II)} :=\|C(v)G_B(\omega_u-\omega_v)C^{-1}(v) - G_B(\omega_u-\omega_v)\|_{\mathscr{L}(V)}.
	\end{align*}
For the term $\mathrm{(I)}$ we have by using the definition \eqref{eq:C} of the family $(C(t), t\in[0,T])$ that
	\begin{equation}
	\label{eq:holder_I}
		\mathrm{(I)}\hspace{-0.5mm} = \hspace{-0.5mm}\frac{1}{p!} \left|\dot{[\omega]}_p^\pi(u)-\dot{[\omega]}_p^\pi(v)\right| \hspace{-0.5mm}\left\|B^p G_B(\omega_u-\omega_v)C^{-1}(v)\right\|_{\mathscr{L}(V)} \leq \frac{k_2 k_\omega}{p!} \sup_{r\in [0,T]}\|B^pC^{-1}(r)\|_{\mathscr{L}(V)}\, |u-v|^{\gamma_2}
	\end{equation}
where we have used the fact that $B^p$ and $G_B$ commute on $\Dom B^p$, H\"older continuity of the derivative $\dot{[\omega]}^\pi_p$ from assumption \ref{ass:omega_2} and also the fact that $B^pC^{-1}(v)$ is a bounded linear operator on $V$. This last claim follows by the closed graph theorem since we have that $\Dom B^p\supset \Dom C$ which implies that the operator $B^pC^{-1}(v)$, that is closed by assumption \ref{ass:B_2}, is defined on the whole space $V$. In order to treat the term $\mathrm{(II)}$ note first that by \cite[equality (1.2) in chapter 3]{Pazy} there is the equality
	\begin{equation*}
		[G_{B+L(t)}(x) - G_B(x)]z = \int_0^xG_{B}(x-\xi)L(t)G_{B+L(t)}(\xi)z\d{\xi}
	\end{equation*}
which holds for every $x>0$ and $z\in\Dom B$. Consequently, there is the estimate
	\begin{equation}
	\label{eq:est_difference_GB}
		\left\|G_{B+L(t)}(x) - G_B(x)\right\|_{\mathscr{L}(V)} \leq \|L(t)\|_{\mathscr{L}(V)} \sup_{|\xi|\leq |x|}\|G_{B}(\xi)\|_{\mathscr{L}(V)} \sup_{|\xi|\leq |x|}\left\|G_{B+L(t)}(\xi)\right\|_{\mathscr{L}(V)} |x|
	\end{equation}
for $x\in\R$. Denote $k_L:=\sup_{r\in [0,T]}\|L(r)\|_{\mathscr{L}(V)}$. By using equality \eqref{eq:representation_of_G(B+L)}, estimate \eqref{eq:est_difference_GB}, and H\"older continuity of the path $\omega$ from assumption \ref{ass:omega_1} successively, we obtain the following:
	\begin{align*}
		\mathrm{(II)} \leq k_\omega k_Lk_1 M\mathrm{exp}\left\{(m+Mk_L)\sup{\Omega_{0,T}}\right\}\,|u-v|^{\gamma_1}.\numberthis\label{eq:Holder_II}
	\end{align*}
Altogether, we obtain from inequalities \eqref{eq:holder_I} and \eqref{eq:Holder_II} that there are constants $\kappa_1, \kappa_2 >0$ such that the inequality
	\begin{equation*}
		\mathrm{(I)} + \mathrm{(II)} \leq \kappa_1|u-v|^{\gamma_1} + \kappa_2|u-v|^{\gamma_2}
	\end{equation*}
holds. This proves that the system $(\tilde{C}_s(t), t\in [0,T])$ satisfies \ref{ass:P3}. 
\medskip\\
Before continuing with the next step note that $\tilde{C}_s(s)=C(s)$ so that this operator is an infinitesimal generator of an analytic semigroup $S_{\tilde{C}_s(s)}=S_{C(s)}$ of bounded linear operators acting on the space $V$. Moreover, for $\alpha\geq 0$ we have $(-\tilde{C}_s(s))^\alpha=(-C(s))^\alpha$ with the domain $\Dom (-\tilde{C}_s(s))^\alpha = D_{\alpha}(s)$. In particular, we have that $\Dom (-\tilde{C}_s(s)) = D_1(s) = \Dom C$. Endow the space $D_1(s)$ with the graph norm of the operator $-C(s)$. We denote this norm by $\|\cdot\|_{D_1(s)}$. 
\medskip\\
Fix $x_0\in V$ and assume that there exists $\varepsilon>0$ is such that $x_0\in D_{1+\varepsilon}(s)$. These will be fixed for the rest of the proof. Define a function $f_{s,x_0}: [s,T]\rightarrow V$ by
	\begin{equation*}
		f_{s,x_0}(t):= \left[\tilde{C}_s(t)-\tilde{C}_s(s)\right]S_{\tilde{C}_s(s)}(t-s)x_0.
	\end{equation*}

\textit{Step 3:} In this step it will show that $f_{s,x_0}$ belongs to the space $\mathscr{C}^{\gamma}([s,T];V)$ where $\gamma=:\min\{\varepsilon, \gamma_1,\gamma_2\}$. To this end, let $u,v\in [s,T]$ be such that $u>v$. We have the inequality 
	\begin{align*}
		\left\|f_{s,x_0}(u)-f_{s,x_0}(v)\right\|_V &\leq \|[\tilde{C}_s(u)-\tilde{C}_s(v)]S_{\tilde{C}_s(s)}(u-s)x_0\|_V \\
		& \hspace{1cm} + \|[\tilde{C}_s(v)-\tilde{C}_s(s)][S_{\tilde{C}_s(s)}(u-s)-S_{\tilde{C}_s(r)}(v-s)]x_0\|_V.
	\end{align*}
Denote the first and the second term on the left-hand side of the above inequality by $I_1$ and $I_2$, respectively. For $I_1$, we have that there exit constants $\tilde{\kappa}_1, \tilde{\kappa}_2>0$ such that 
	\begin{equation}
	\label{eq:Holder_I1}
		I_1 \leq \|\tilde{C}_s(u)-\tilde{C}_s(v)\|_{\mathscr{L}(D_1(s);V)} \|S_{\tilde{C}_s(r)}(u-r)x_0\|_{D_1(s)} \leq \tilde{\kappa}_1(u-v)^{\gamma_1} + \tilde{\kappa}_2(u-v)^{\gamma_2}
	\end{equation}
by H\"older continuity of the family $(\tilde{C}_s(t),t\in[0,T])$ proved in \textit{Step 2} and by continuity of the function $\tau\mapsto S_{\tilde{C}_s(r)}(\tau-r)x_0$ on the interval $[r,T]$ in the norm $\|\cdot\|_{D_1(s)}$. Similarly, we have the following estimate for $I_2$:
	\begin{align*}
		I_2 \leq \|\tilde{C}_s(v)-\tilde{C}_s(s)\|_{\mathscr{L}(D_1(s);V)}\|[S_{\tilde{C}_s(s)}(u-s)-S_{\tilde{C}_s(s)}(v-s)]x_0\|_{D_1(s)}.
	\end{align*}
The term $\|\tilde{C}_s(v)-\tilde{C}_s(s)\|_{\mathscr{L}(D_1(s);V)}$ is bounded by $\kappa_1T^{\gamma_1}+\kappa_2T^{\gamma_2}$ by H\"older continuity of the family $(\tilde{C}_s(t), t\in [0,T])$ and for the second term, we obtain
	\begin{align*}
		\|[S_{\tilde{C}_s(s)}(u-s) - S_{\tilde{C}_s(s)}(v-s)]x_0\|_{D_1(s)} & = \\
				& \hspace{-5cm} = \|[S_{\tilde{C}_s(s)}(u-s) - S_{\tilde{C}_s(s)}(v-s)]x_0\|_V + \|(-\tilde{C}_s(s))[S_{\tilde{C}_s(s)}(u-s) - S_{\tilde{C}_s(s)}(v-s)]x_0\|_V \\
				& \hspace{-5cm} \leq (u-v)\sup_{\tau\in [0,T]}\|\tilde{C}_s(s)S_{\tilde{C}_s(s)}(\tau)x_0\|_V + C_\varepsilon (u-v)^\varepsilon \sup_{\tau\in [0,T]}\|S_{\tilde{C}_s(s)}(\tau)\|_{\mathscr{L}(V)}\|(-\tilde{C}_s(s))^{1+\varepsilon}x_0\|_V
	\end{align*}
with some constant $C_\varepsilon>0$ by using successively the definition of the graph norm $\|\cdot\|_{D_1}$, \cite[Theorem 2.4, d), in chapter 1]{Pazy}, commutativity of the semigroup $S_{\tilde{C}_s(s)}$ with its generator $\tilde{C}_s(s)$ on $D_1$ ensured by \cite[Theorem 2.4, b), in chapter 1]{Pazy} (together with the fact that $D_{1+\varepsilon}\subseteq D_1$ that holds by \cite[Theorem 6.8, b), in chapter 2]{Pazy}), \cite[Theorem 6.13, d), in chapter 2]{Pazy}, and commutativity of the semigroup $S_{\tilde{C}_s(s)}$ with its generator $\tilde{C}_s(s)$ on $D_{1+\varepsilon}$ ensured by \cite[Theorem 6.13, b), in chapter 2]{Pazy}. Thus it is seen that there exist constants $c_1,c_2>0$ such that 
	\begin{equation}
	\label{eq:Holder_I2}
		I_2 \leq c_1 (u-v) + c_2 (u-v)^\varepsilon.
	\end{equation}
It follows from the inequalities \eqref{eq:Holder_I1} and \eqref{eq:Holder_I2} that 
	\begin{equation*}
		\|f_{s,x_0}(u)-f_{s,x_0}(v)\|_V \leq \tilde{\kappa}_1 (u-v)^{\gamma_1} + \tilde{\kappa}_2(u-v)^{\gamma_2} + c_1 (u-v) + c_2(u-v)^\varepsilon.
	\end{equation*} 
which shows that indeed $f_{s,x_0}\in \mathscr{C}^\gamma([s,T];V)$.

\bigskip

\textit{Step 4:} We have shown that the family $(\tilde{C}_s(t), t\in [0,T])$ satisfies conditions \ref{ass:P1} - \ref{ass:P3} and that the function $f_{s,x_0}$ is such that it belongs to the space $\mathscr{C}^{\gamma}([s,T];V)$. Moreover, since it clearly satisfies $f_{s,x_0}(s)=0$, it follows by \cite[Theorem 7.1 and Remark 7.2 in chapter 5]{Pazy} that there is a (unique) function $u_{s,x_0}: [s,T]\rightarrow V$ such that $u_{s,x_0}\in \mathscr{C}^{1,\gamma'}([s,T];V)$ with $\gamma'<\gamma$; $u_{s,x_0}(t)\in \Dom C$ for every $t\in [s,T]$ (we can include the point $s$ in the interval since $0\in \Dom C$); and such that it satisfies the equation 
	\begin{equation}
	\label{eq:inhomoeneous_NCP}
		\begin{cases}
			\dot{u}(t) & = \tilde{C}_s(t)u(t) + f_{s,x_0}(t), \quad s<t\leq T,\\
			u(s) & = 0.
		\end{cases}
	\end{equation}
If we now define a function $v_{s,x_0}: [s,T]\rightarrow V$ by
	\begin{equation}
	\label{eq:smooth_solution}
		v_{s,x_0}(t):= u_{s,x_0}(t) + S_{\tilde{C}_s(s)}(t-s)x_0, 
	\end{equation}
it follows that ${v}_{s,x_0}\in\mathscr{C}^1([s,T];V)$, $v_{s,x_0}(t)\in \Dom C$ for every $t\in [s,T]$, and it is easily verified that $v_{s,x_0}$ satisfies the equation
	\begin{equation*}
		\begin{cases}
			\dot{v}_{s,x_0}(t) & = \tilde{C}_s(t)v_{s,x_0}(t), \quad s<t\leq T,\\
			v_{s,x_0}(s) & = x_0,
		\end{cases}
	\end{equation*}
so that $v_{s,x_0}$ is the $\Dom C$-valued solution to problem \eqref{eq:NCP}. The proof is concluded by appealing to \autoref{prop:linear_equation_without_commutativity}.
\end{proof}

\begin{remark}
\label{rem:structure_of_parabolic_sol}
Let us comment on the structure of the strong solution to \eqref{eq:BCP} in \autoref{prop:parabolic_noncom}. Set $\Delta(T):=\{(t,r)\in [0,T]^2\,|\, t\geq r\}$.
Since the family of operators $(\tilde{C}_s(r), r\in [0,T])$ satisfies conditions \ref{ass:P1} - \ref{ass:P3}, there exists a unique strongly continuous evolution system $U_s$ corresponding to this family by \cite[Theorem 6.1 in chapter 5]{Pazy}. That is, there exists a function $U_s:\Delta(T)\rightarrow\mathscr{L}(V)$ such that it satisfies the \textit{evolution property}
	\begin{align*}
		& U_s(r,r)  = \mathrm{Id}_V,  \quad\quad\quad\quad\quad\hspace{3mm} 0\leq r\leq T,\\
		& U_s(t,r)  = U_s(t,x)U_s(x,r),  \quad 0\leq r\leq x\leq t\leq T;
	\end{align*}
is \textit{strongly continuous}, i.e. for every $v\in V$ the function $U_s(\cdot,\cdot)v: \Delta(T)\rightarrow V$ is continuous; and such that it \textit{corresponds} to the family $(\tilde{C}_s(t), t\in [0,T])$, i.e. the following holds:
	\begin{enumerate}[label=(\roman*)]
	\itemsep0em
		\item There exists a constant $c_s>0$ such that $\|U_s(t,r)\|_{\mathscr{L}(V)}\leq c_s$ is satisfied for every $(t,r)\in\Delta(T)$.
		\item For every $0\leq r<t\leq T$, the range of $U_s(t,r)$ is the space $\Dom C$.
		\item For every $0\leq r<t\leq T$, the derivative $(\partial_1U_s)(t,r)$ is a bounded linear operator on the space $V$, it is strongly continuous on $0\leq s<t\leq T$ and satisfies 
			\begin{equation*}
				(\partial_1U_s)(t,r) = \tilde{C}_s(t)U_s(t,r), \quad r<t\leq T.
			\end{equation*}
		\item For every $v\in \Dom C$ and $t\in (0,T]$, the function $U_s(t,\cdot)v: [0,t]\rightarrow V$ is differentiable and satisfies
			\begin{equation*}
				[\partial_2U_s(t,\cdot)v](r) = - U_s(t,r)\tilde{C}_s(r)v, \quad 0\leq r\leq t.
			\end{equation*}
	\end{enumerate}
It is immediate (or by \cite[Theorem 6.8 in chapter 5]{Pazy}) that for every $x_0\in V$, the function $\tilde{v}_{s,x_0}$ defined by $\tilde{v}_{s,x_0}(t):=U_s(t,s)x_0$ is the unique function that belongs to the space $\mathscr{C}^1((s,T];V)\cap \mathscr{C}([s,T];V)$ and satisfies equation \eqref{eq:NCP}. This function also satisfies $\tilde{v}_{s,x_0}(t)\in\Dom C$ for every $t\in (s,T]$ and if we assume that $x_0\in \Dom C$ then we also have that $v_{s,x_0}(s)\in\Dom C$. The argument in \textit{Step 4} of the proof of \autoref{prop:parabolic_noncom} was needed in order to find a function $v_{s,x_0}$ that satisfies equation \eqref{eq:NCP} and that belongs to the space $\mathscr{C}^1([s,T];V)$. Such function exists under the stronger assumption that $x_0\in D_{1+\varepsilon}(s)$ for some $\varepsilon>0$ and it is given by formula \eqref{eq:smooth_solution}. However, by uniqueness of the solution to the problem \eqref{eq:NCP}, we have that $\tilde{v}_{s,x_0}\equiv v_{s,x_0}$. This means that the strong solution $X_{s,x_0}^\omega$ to the problem \eqref{eq:BCP} is in fact given by 
	\begin{equation*}
		X_{s,x_0}^\omega(t)= G_B(\omega_t-\omega_s)U_s(t,s)x_0, \quad t\in [s,T].
	\end{equation*}
\end{remark}

\begin{example}
\label{ex:parabolic}
Consider the formal equation
	\begin{equation*}
		(\partial_t u)(t,x) = (\partial_x^2u)(t,x) + g(x)(\partial_xu)(t,x)\dot{\omega}_t
	\end{equation*}
for $(t,x)\in [0,1]\times\mathbb{R}$ with the initial condition $u(0,x)=u_0(x)$ for $x\in\R$ where $g\in\mathscr{C}^3_b(\R)$ is a given function. Assume further that $g$ is such that $0<c_1\leq g(x)\leq c_2<\infty$ holds for every $x\in\R$ and some constants $c_1,c_2$. It is moreover assumed that $\omega$ is a function that belongs to the space $\mathscr{C}^{\gamma_1}([0,1])$ and that has finite $4$-th variation along some sequence of partitions $\pi\in\mathcal{P}[0,1]$ whose mesh size tends to zero. It is also assumed that $[\omega]_4^\pi\in\mathscr{C}^{1,\gamma_2}([0,1])$ for some $\gamma_2\in(0,1)$ so that the function $\omega$ satisfies conditions \ref{ass:omega_1} and \ref{ass:omega_2}. Additionally, we assume that the derivative satisfies $\dot{[\omega]}_4^\pi(t)>0$ for every $t\in [0,1]$ (recall that such functions $\omega$ naturally arise as paths of Wiener integrals with respect to a singular fractional Brownian motion, cf. \autoref{ex:stoch_int_fBm} and \autoref{ex:integrands_have_variation}).\medskip\\
The formal equation can be given rigorous meaning in our framework as follows. Set $V:=L^2(\R)$, $A:=\partial^2$ on $\Dom A:=W^{2,2}(\R)$, $B:=g\partial$ on $\Dom B:=W^{1,2}(\R)$ (where $[(g\partial)f](x)=g(x)f'(x)$ for $f\in W^{1,2}(\R)$). The operator $g\partial$ generates a strongly continuous group $G_{g\partial}$ of bounded linear operators on the space $L^2(\R)$ by \autoref{lem:A1} and \autoref{lem:A2} so that \ref{ass:B_1} is satisfied. Moreover, we have that the operator $B^4=(g\partial)^4$ has the domain $\Dom B^4=W^{4,2}(\R)$ and is given by
	\begin{equation*}
		(g\partial)^4f = q_1 f'+ q_2f'' + q_3f''' + q_4f^{(4)}
	\end{equation*}
for $f\in W^{4,2}(\R)$ where 
	\begin{align*}
		q_1(x) & := g(x)[g'(x)]^3 + [g(x)]^3g'''(x) + 4[g(x)]^2g'(x)g''(x)\\
		q_2(x) & := 7[g(x)]^2[g'(x)]^2 + 4[g(x)]^3 g''(x)\\
		q_3(x) & := 6[g(x)]^3g'(x)\\
		q_4(x) & := [g(x)]^4
	\end{align*}
for $x\in\R$. This operator is closed so that \ref{ass:B_2} is satisfied. We have that $\Dom C=W^{4,2}(\R)$ and
	\begin{equation*}
		C(t) = -\frac{1}{4!}\dot{[\omega]}_4^\pi(t) q_1\partial + \left(1-\frac{1}{4!}\dot{[\omega]}_4^\pi(t)q_2\right)\partial^2 -\frac{1}{4!}\dot{[\omega]}_4^\pi(t)q_3\partial^3 -\frac{1}{4!}\dot{[\omega]}_4^\pi(t)q_4\partial^4
	\end{equation*}
for every $t\in [0,1]$. Since $W^{4,2}(\R)$ is dense in $L^2(\R)$ and $C(t)$ is closed for every $t\in [0,1]$, we have that condition \ref{ass:C_1} is satisfied. Since we have that $\dot{[\omega]}_4^\pi>0$ and $0<c_1\leq g\leq c_2<\infty$, we obtain that $-C(t)$ is a strongly elliptic operator and, consequently, $C(t)$ generates an analytic semigroup on the space $L^2(\R)$. Hence, by \cite[Theorem 5.2 in chapter 2]{Pazy}, it follows that there exists $\delta(t)\in (0,\sfrac{\pi}{2})$ and $\lambda_0(t)\in\mathbb{R}$ such that the resolvent $\rho(C(t))$ contains the sector $\Sigma(t)\cup\{\lambda_0(t)\}$ where
	\begin{equation*}
		\Sigma(t):=\left\{\lambda\in\mathbb{C}\,\left|\, |\arg (\lambda-\lambda_0(t))| <\frac{\pi}{2}+\delta(t)\right.\right\}
	\end{equation*}
and, moreover, there exists a constant $m(t)>0$ such that the inequality 
	\begin{equation*}
		\|R(\lambda:C(t))\|_{\mathscr{L}(V)}\leq \frac{m(t)}{|\lambda-\lambda_0(t)|}, \quad \lambda\in \Sigma(t),
	\end{equation*}
is satisfied. Hence, by \cite[Remark 3.3.2]{Tan79}, we have that the half-plane $\mathbb{C}_{\lambda_0(t)}^+$ is contained in $\rho(C(t))$ and the inequality
	\begin{equation*}
		\|R(\lambda:C(t))\|_{\mathscr{L}(V)} \leq \frac{\tilde{m}(t)}{1+|\lambda-\lambda_0(t)|}
	\end{equation*}
is satisfied for every $\lambda\in\mathbb{C}_{\lambda_0(t)}^+$. Hence, assumption \ref{ass:C_2} is satisfied with \[k_3:=\max_{t\in [0,T]}\tilde{m}(t)\quad\mbox{and}\quad\lambda_{0}:=\max_{t\in [0,T]}\lambda_0(t).\] It remains to show that assumption \ref{ass:BC} is satisfied as well. Consider the operator $T_g:=g''\partial + 2g'\partial^2$ on $\Dom T_g:=W^{2,2}(\R)$ (i.e. $[T_gf](x)=g''(x)f'(x) + 2g'(x)f''(x)$ for $f\in W^{2,2}(\R)$). It is straightforward to check that the equality
	\begin{equation*}
		B(\lambda -C(t))v = (\lambda-C(t))Bv + T_qv
	\end{equation*}
holds for every $t\in [0,1]$, $\lambda\in\rho(C(t))$ and $v\in W^{5,2}(\R)$. Let $\vartheta\in\R $ be such that $\vartheta>\lambda_0$. Then we have that for every $t\in [0,1]$ and every $\tilde{v}\in W^{1,2}(\R)$, $(\vartheta-C(t))^{-1}\tilde{v}$ belongs to the space $W^{5,2}(\R)$ and the equality
	\begin{equation*}
		B\tilde{v} = (\vartheta - C(t))B(\vartheta-C(t))^{-1}\tilde{v} + T_g(\vartheta-C(t))^{-1}\tilde{v}
	\end{equation*}
holds. Consequently, assumption \ref{ass:BC} is satisfied with $\theta:=\vartheta$, $K:=W^{1,2}(\R)$ and the family of operators $(L(t), t\in [0,1])$ defined by $L(t):= -T_g(\vartheta-C(t))^{-1}$ for $t\in [0,1]$. 
Thus the conditions of \autoref{prop:parabolic_noncom} are verified and therefore, if $x_0\in D_{1+\varepsilon}(0)$ for some $\varepsilon>0$, then there is a strong solution to the problem 
	\begin{equation*}
		\begin{cases}
			\d{X}_t & = \partial^2X_t + g\partial X_t\d^\pi\omega_t, \quad 0<t\leq 1,\\
			\phantom{\d}X_0 & = x_0.
		\end{cases}
	\end{equation*}
\end{example}

\subsubsection{The hyperbolic case}

In what follows, we give some sufficient conditions for the solvability of the system \eqref{eq:NCP} in what is usually called the hyperbolic case. For the purposes of the following lines, assume that there exists a Hilbert space $Y$ that is continuously and densely embedded in the Hilbert space $V$. We say that a family $(K(t), t\in [0,T])$ of infinitesimal generators of strongly continuous semigroups is \textit{hyperbolic} if the following three conditions are satisfied:

\begin{enumerate}[label=(H\arabic*)]
	\itemsep0em
	\item\label{ass:H1} The family $(K(t), t\in [0,T])$ is $(M_K,\omega_K)$-stable in $V$, i.e. there exists constants $M_K\geq 1$ and $\omega_k\in\R$ such that $(\omega_k,\infty)\subset\rho(K(t))$ holds for every $t\in [0,T]$ and the inequality
		\begin{equation*}
			\left\|\prod_{i=1}^k R(\lambda:K(t_i)) \right\|_{\mathcal{L}(V)} \leq \frac{M_K}{(\lambda - \omega_k)^k}
		\end{equation*}
	is satisfied for every $\lambda >\omega_K$ and every finite sequence $0\leq t_1\leq t_2\leq \ldots \leq t_k\leq T$, $k=1,2, \ldots $. 
	\item\label{ass:H2} There exists a family $(Q(t), t\in [0,T])$ of isomorphic mappings from $Y$ onto $V$ such that for every $v\in Y$,  the function $Q(\cdot)v$ is continuously differentiable in the space $V$ on the interval $[0,T]$ and there exists a strongly continuous function $k:[0,T]\rightarrow\mathscr{L}(V)$ such that 
		\begin{equation*}
			Q(t)K(t)Q(t)^{-1}v = K(t)v + k(t)v
		\end{equation*}
	holds for every $v\in \Dom (K(t))$ and $t\in [0,T]$.
	\item\label{ass:H3} For every $t\in [0,T]$, it holds that $Y\subseteq\Dom K(t)$. Moreover, the map $K: [0,T]\rightarrow \mathscr{L}(Y;V)$ is norm-continuous. 
\end{enumerate}

\begin{proposition}
\label{prop:hyperbolic_noncom}
Assume that $\omega$ satisfies the condition
	\begin{enumerate}[label=(h\arabic*)]
		\itemsep0em
		\item\label{ass:h0} The $p$-th variation $[\omega]^\pi_p$ of the path $\omega$ is continuously differentiable on the interval $[0,T]$ with the derivative being denoted by $\dot{[\omega]}^\pi_p$.
	\end{enumerate}
Assume that the operator $B$ satisfies the condition
	\begin{enumerate}[label=(h\arabic*)]
	\setcounter{enumi}{1}
		\itemsep0em
		\item\label{ass:h1} The operator $B$ is an infinitesimal generator of a strongly continuous group of bounded linear operators $G_B$ acting on the space $V$.
		\item\label{ass:h11} For every $x\in\mathbb{R}$ it holds that $\|G_B(x)\|_{\mathscr{L}(V)}\leq 1$.
	\end{enumerate}
Assume further that there exists a Hilbert space $Y$ that is continuously and densely embedded in $V$ such that the operator $B$ satisfies the condition
	\begin{enumerate}[label=(h\arabic*)]
		\itemsep0em
		\setcounter{enumi}{3}
		\item\label{ass:h3} The space $Y$ is closed under the action of the group $G_B$, i.e. $G_BY\subseteq Y$.
		\item\label{ass:h2} For every $x\in\mathbb{R}$ it holds that $\|G_B(x)\|_{\mathscr{L}(Y)}\leq 1$.
	\end{enumerate}
Moreover, it is assumed that the family $(C(t), t\in [0,T])$ defined by \eqref{eq:C} satisfies the conditions
	\begin{enumerate}[label=(h\arabic*)]
		\itemsep0em
		\setcounter{enumi}{5}
		\item\label{ass:h4} $Y\subseteq \Dom C$ and the operator $C(t)$ is closed for every $t\in [0,T]$.
		\item\label{ass:h5} There exists a constant $\lambda_1$ such that the resolvent set $\rho(C(t))$ contains the ray $\mathbb{R}_{\lambda_1}^+$ for every $t\in [0,T]$ and such that the inequality
			\begin{equation*}
				\|R(\lambda: C(t))\|_{\mathscr{L}(V)} \leq \frac{1}{\lambda-\lambda_1}
			\end{equation*}
		is satisfied for every $\lambda>\lambda_1$ and every $t\in [0,T]$. 
		\item\label{ass:h6} There exists an isomorphism $Q\in\mathscr{L}(Y;V)$ such that 
			\begin{enumerate}
				\item for every $t\in [0,T]$, it holds that $C(t)Q^{-1}(\Dom C)\subseteq Y$;
				\item for every $x\in\R$ and every $y\in Y$, it holds that
					\begin{equation*}
						QG_B(x)y = G_B(x)Qy;
					\end{equation*}
			\end{enumerate}
		and there exists a strongly continuous function $L:[0,T]\rightarrow\mathscr{L}(V)$ such that the equality
			\begin{equation*}
				QC(t)Q^{-1}v = C(t)v + L(t)v
			\end{equation*}
		is satisfied for every $t\in [0,T]$ and every $v\in \Dom C$. 
	\end{enumerate}
Then for every $s\in [0,T)$ and $x_0\in Y$, there exists a strong solution to the bilinear problem \eqref{eq:BCP} that takes values in the space $Y$.
\end{proposition}

\begin{proof}
Let $s\in [0,T)$ be fixed and define the family of linear operators
	\begin{equation*}
		\tilde{C}_s(t):= z_{s,\omega}^{-1}(t)C(t)z_{s,\omega}(t), \quad t\in [0,T],
	\end{equation*}
on the domain $\Dom\tilde{C}_s(t):=z_{s,\omega}^{-1}(t)\left(\Dom C\right)$. 
\medskip

\textit{Step 1:} It is shown first that the family $(\tilde{C}_s(t), t\in [0,T])$ is stable in $V$, i.e. that this family satisfies condition \ref{ass:H1}. Indeed, it follows that for every $t\in [0,T]$, the space $\Dom \tilde{C}_s(t)$ is dense in $V$ since $Y$ is dense in $V$ and $Y\subseteq \Dom C\subseteq \Dom \tilde{C}_s(t)$ by assumption \ref{ass:h4}. Moreover, the operator $\tilde{C}_s(t)$ is closed because $C(t)$ is closed by the same assumption. Furthermore, as in the proof of \autoref{prop:parabolic_noncom}, there is the inclusion $\rho(C(t))\subseteq \rho(\tilde{C}_s(t))$ and the equality
	\begin{equation}
	\label{eq:resolvent_of_C_s(t)_hyperbolic}
		R(\lambda:\tilde{C}_s(t)) = z_{s,\omega}^{-1}(t)R(\lambda: C(t))z_{s,\omega}(t)
	\end{equation}
is satisfied for every $\lambda \in\rho(C(t))$. Hence, it follows by assumptions \ref{ass:h11} and \ref{ass:h5} that $\mathbb{R}_{\lambda_1}^+\subset\rho(\tilde{C}_s(t))$ and, moreover, there is the estimate
	\begin{equation*}
		\|R(\lambda:\tilde{C}_s(t))\|_{\mathscr{L}(V)} \leq \|z_{s,\omega}(t)\|_{\mathscr{L}(V)}\|R(\lambda:C(t))\|_{\mathscr{L}(V)}\|z_{s,\omega}(t)\|_{\mathscr{L}(V)} \leq \frac{1}{\lambda-\lambda_1}
	\end{equation*}
for every $\lambda>\lambda_1$. Therefore, by the Hille-Yosida theorem \cite[Corollary 3.8 in chapter 1]{Pazy}, it follows that the operator $\tilde{C}_s(t)$ generates a strongly continuous semigroup $S_{\tilde{C}_s(t)}$ of bounded linear operators acting on $V$ such that $\|S_{\tilde{C}_s(t)}(u)\|_{\mathscr{L}(V)}\leq \mathrm{e}^{\lambda_1u}$ holds for every $u\geq 0$. Consequently, the family $(\tilde{C}_s(t), t\in [0,T])$ of generators is $(1,\lambda_1)$-stable in $V$ by the remark preceding \cite[Theorem 2.2 in chapter 5]{Pazy}.
\medskip

\textit{Step 2:} It is shown that there exists a strongly continuous map $\tilde{L}: [0,T]\rightarrow \mathscr{L}(V)$ such that 
	\begin{equation*}
		Q\tilde{C}_s(t)Q^{-1}v = \tilde{C}_s(t)v + \tilde{L}(t)v
	\end{equation*}
holds for every $t\in [0,T]$ and every $v\in \Dom\tilde{C}_s(t)$, i.e. that the family $(\tilde{C}_s(t), t\in [0,T])$ satisfies condition \ref{ass:H2}. To this end, let $t\in [0,T]$ and $v\in \Dom\tilde{C}_s(t)$ be arbitrary. Then, there exists a $\tilde{v}\in \Dom C$ such that $v=G_B^{-1}(\omega_t-\omega_s)\tilde{v}$ and we have that 
	\begin{align*}
		Q\tilde{C}_s(t)Q^{-1}v & = Q\tilde{C}_s(t)Q^{-1}G_B^{-1}(\omega_t-\omega_s)\tilde{v}\\
			& = Q\tilde{C}_s(t)G_B^{-1}(\omega_t-\omega_s)Q^{-1}\tilde{v}\\
			& = QG_B^{-1}(\omega_t-\omega_s)C(t)Q^{-1}\tilde{v}\\
			& = G_B^{-1}(\omega_t-\omega_s)QC(t)Q^{-1} \tilde{v}\\
			& = G_B^{-1}(\omega_t-\omega_s)[C(t) + L(t)]\tilde{v}\\
			& = G_B^{-1}(\omega_t-\omega_s)C(t)G_B(\omega_t-\omega_s)v + G_B^{-1}(\omega_t-\omega_s)L(t)G_B(\omega_t-\omega_s)v\\
			& = \tilde{C}_s(t)v + G_B^{-1}(\omega_t-\omega_s)L(t)G_B(\omega_t-\omega_s)v
	\end{align*}
by using assumption \ref{ass:h6}.
\medskip

\textit{Step 3:} Finally, it is shown that there exists a constant $K>0$ such that $\|\tilde{C}_s(t)\|_{\mathscr{L}(Y;V)}\leq K$ holds for every $t\in [0,T]$, i.e. that the family satisfies condition \ref{ass:H3}. For $v\in Y$, we have 
	\begin{align*}
		\|\tilde{C}_s(t)v\|_{V} & = \|z_{s,\omega}^{-1}(t)C(t)z_{s,\omega}(t)\|_{V}\\
			& \leq \|z_{s,\omega}^{-1}(t)\|_{\mathscr{L}(V)} \|C(t)z_{s,\omega}(t)v\|_{V}\\
			& \leq \|C(t)\|_{\mathscr{L}(Y;V)} \|z_{s,\omega}(t)v\|_{Y} \\
			& \leq \|C(t)\|_{\mathscr{L}(Y;V)}\|z_{s,\omega}(t)\|_{\mathscr{L}(Y)}\|v\|_{Y}\\
			& \leq \sup_{t\in [0,T]} \|C(t)\|_{\mathscr{L}(Y;V)} \|v\|_{Y}.
	\end{align*}
by appealing to assumption \ref{ass:h2}. The supremum is finite by the closed graph theorem since for every $t\in [0,T]$, the operator $C(t)$ is closed and defined on the whole space $Y$.
\medskip

\textit{Step 4:} From \textit{Step 1} - \textit{Step 3} it follows by \cite[Theorem 4.5.1 and Theorem 4.5.2]{Tan79} that for every $x_0$ there exists a unique function $v_{s,x_0}$ that belongs to the space $\mathscr{C}([s,T];Y)\cap\mathscr{C}^1([s,T];V)$ and that satisfies equation \eqref{eq:NCP}. Moreover, since by assumptions \ref{ass:h0} and \ref{ass:h1}, \autoref{ass:1} is satisfied, we have by \autoref{prop:linear_equation_without_commutativity} that the problem \eqref{eq:BCP} admits a strong solution.
\end{proof}

\begin{remark}
\label{rem:structure_of_hyperbolic_sol}
Let us comment on the structure of the strong solution to \eqref{eq:BCP} in \autoref{prop:hyperbolic_noncom}. By \cite[Theorem 4.4.2]{Tan79}, there exists a unique strongly continuous evolution system $U_s$ on $\Delta(T)$ with the following properties:
	\begin{enumerate}[label=(\roman*)]
	\itemsep0em
		\item For every $(t,r)\in\Delta(T)$, it holds that $\|U_s(t,r)\|_{\mathscr{L}(V)}\leq \mathrm{e}^{\lambda_1(t-r)}.$
		\item For every $0\leq r<t\leq T$, it holds that $U_s(t,r)Y\subseteq Y$ and $U_s(t,r)$ is strongly continuous in $Y$.
		\item For every $v\in Y$ and $r\in [0,T)$, the function $U_s(t,r)v$ is strongly continuously differentiable in the first variable and it holds that  
			\begin{equation*}
				[\partial_1U_s(\cdot,r)v](t) = \tilde{C}_s(t)U_s(t,r)v, \quad r\leq t\leq T.
			\end{equation*}
		\item For every $v\in Y$ and $t\in (0,T]$, the function $U_s(t,r)v$ is strongly continuously differentiable in the second variable and it holds that 
			\begin{equation*}
				[\partial_2U_s(t,\cdot)v](r) = - U_s(t,r)\tilde{C}_s(r)v, \quad 0\leq r\leq t.
			\end{equation*}
	\end{enumerate}
Consequently, by \cite[Theorem 4.5.1 and Theorem 4.5.2]{Tan79}, we have that for $x_0\in Y$, the solution to \eqref{eq:NCP} is unique and it is given by $v_{s,x_0}(t)=U_s(t,s)x_0$. Consequently, we have that the strong solution $X_{s,x_0}^\omega$ to the problem \eqref{eq:BCP} constructed in \autoref{prop:hyperbolic_noncom} is in fact given by 
	\begin{equation*}
		X_{s,x_0}^\omega(t)= G_B(\omega_t-\omega_s)U_s(t,s)x_0, \quad t\in [s,T].
	\end{equation*}
\end{remark}

\begin{example}
\label{ex:hyperbolic}
Consider the formal equation 
	\begin{equation*}
		(\partial_t u)(t,x) = \frac{1}{8}(\partial_x^4u)(t,x) + \mathrm{i}\left[(\partial_x^4u)(t,x)+ (\partial_x^2u)(t,x) - g(x)u(t,x)\right] + (\partial_xu)(t,x)\dot{W}^{\frac{1}{4}}_t
	\end{equation*}
for $(t,x)\in [0,1]\times \mathbb{R}$ with the initial condition $u(0,x) = u_0(x)$ for $\in\mathbb{R}$ where $g\in\mathscr{C}^6_b(\R)$ and where $\mathrm{i}$ is the imaginary unit. This formal equation can be given rigorous meaning in our framework as follows. Set $V:=L^2(\R;\mathbb{C})$, $A:=\frac{1}{8}\partial^4+\mathrm{i}(\partial^4 + \partial^2 - g\,\mathrm{Id})$ on $\Dom A=W^{4,2}(\R;\mathbb{C})$, $B:=\partial$ on $\Dom B=W^{1,2}(\R;\mathbb{C})$, $\omega:=W^{\frac{1}{4}}$ is the fractional Brownian motion with the Hurst parameter $H=\sfrac{1}{4}$.
\medskip

As before, if we take the sequence of partitions $\pi$ to be either the sequence of dyadic partitions $\tilde{D}$ or the sequence $\tilde{E}$, we have by \autoref{ex:integrands_have_variation} that $[W^{\frac{1}{4}}]^{\pi}_{4}(t) = 3t$ which is continuously differentiable so that assumption \ref{ass:h0} is satisfied. Moreover, as before, we have that $B=\partial$ generates a strongly continuous group $G_{\partial}$ of left-shift operators on $L^2(\R;\mathbb{C})$ so that assumption \ref{ass:h1} is also satisfied and assumption \ref{ass:h11} is easily verified.
\medskip

Set 
	\begin{equation*}
		Y:=W^{6,2}(\R;\mathbb{C}).
	\end{equation*}
Clearly, we have that this space is continuously and densely embedded in $L^2(\R;\mathbb{C})$ and closed under the action of the group $G_{\partial}$ so that assumption \ref{ass:h2} is satisfied and assumption \ref{ass:h3} is also be easily verified.
\medskip

For the family of operators $(C(t), t\in [0,1])$, we have for every $t\in [0,1]$ that $C(t)=C$ where the operator $C$ is defined on $\Dom C= W^{4,2}(\R;\mathbb{C})$ by
	\begin{equation*}
		[Cf](x):= \mathrm{i} \left[f^{(\mathrm{4})}(x) + f{''}(x) -g(x)f(x)\right], \quad f\in W^{4,2}(\R;\mathbb{C}).
	\end{equation*}
Clearly it holds that $Y=W^{6,2}(\R;\mathbb{C}) \subseteq W^{4,2}(\R;\mathbb{C})=\Dom C$. Moreover, since the operator $C_0$ defined by $C_0:=\partial^4 +\partial^2-g$ is self-adjoint on $L^2(\R)$, it follows that $C=\mathrm{i}C_0$ is skew-adjoint on $L^2(\R;\mathbb{C})$ and therefore, it generates a unitary, strongly continuous group of bounded linear operators on $L^2(\R;\mathbb{C})$ by Stone's theorem. Thus, both assumptions \ref{ass:h4} and \ref{ass:h5} are satisfied.
\medskip

Finally, it remains to verify assumption \ref{ass:h6}. To this end, set $Q:=(\vartheta - \partial^6)$ on $W^{6,2}(\R;\mathbb{C})=Y$ where $\vartheta\in\rho(\partial^6)$. Clearly, for $f\in W^{4,2}(\R;\mathbb{C})=\Dom C$, we have that $C(\vartheta-\partial^6)^{-1}f\in W^{6,2}(\R;\mathbb{C})=Y.$ Moreover, we have that 
	\begin{equation*}
		[(\vartheta-\partial^6) G_{\partial}(x)f](z)=\vartheta f(x+z) - f^{(6)}(x+z) = [G_{\partial}(x)(\vartheta - \partial^6)f](z), \quad \mathrm{a. e.}\,\, z\in\mathbb{R},
	\end{equation*}
holds for every $f\in W^{6,2}(\R;\mathbb{C})=Y$. Finally, consider the operator $T_g$ defined by
	\begin{equation*}
		[T_gv](x) = \sum_{k=0}^5{{6}\choose{k}} g^{(6-k)}(x)v^{(k)}(x)
	\end{equation*}
defined for $v\in \Dom T_g=W^{5,2}(\R;\mathbb{C})$. It is straightforward to verify that the equality
	\begin{equation*}
		(\vartheta-\partial^6)Cv= C(\vartheta-\partial^6)v + T_gv
	\end{equation*}
holds for every $v\in W^{10,2}(\R;\mathbb{C})$. Consequently, we have that for $\tilde{v}\in W^{4,2}(\R;\mathbb{C})=\Dom C$, $(\vartheta-\partial^6)^{-1}\tilde{v}\in W^{10,2}(\R;\mathbb{C})$ and the equality  
	\begin{equation*}
		(\vartheta-\partial^6)C(\vartheta-\partial^6)^{-1}\tilde{v} = C\tilde{v} + T_g(\vartheta-\partial^6)^{-1}\tilde{v}
	\end{equation*}
is satisfied. Thus, the conditions of \autoref{prop:hyperbolic_noncom} are verified and therefore, if $u_0\in W^{6,2}(\R;\mathbb{C})$, then there is a strong solution to the problem
	\begin{equation*}
		\begin{cases}
			\d{X}_t  & = \left[\frac{1}{8}\partial^4 + \mathrm{i}\left(\partial^4 + \partial^2 -g\right)\right]X_t\d{t} + \partial X_t\d^{\pi}W_t^{\frac{1}{4}}, \quad 0<t\leq 1,\\
			\phantom{\d}X_0 & = u_0.
		\end{cases}
	\end{equation*}
\end{example}

\section*{Acknowledgement}
The research of P. \v{C}. has been supported by the Czech Science Foundation, project GA\v{C}R 19-07140S. M.J. G.-A. has been supported by Ministerio de Ciencia, Innovaci\'on y Universidades, Grant No. PGC2018-096540-I00. She also would like to thank VIASM in Hanoi (Vietnam) for its great hospitality and support, that helped her to develop part of the research of this paper.

\newpage
\begin{appendices}
\section{Strongly continuous group generated by $g\partial$}
Assume that $g\in\mathscr{C}(\R)$ for which there exists two constants $c_1,c_2$ such that 
	\begin{equation*}
		0<c_1\leq g(x)\leq c_2 <\infty
	\end{equation*}	
holds for every $x\in\R$. In this appendix, it is shown that the operator $g\partial: W^{1,2}(\R)\rightarrow L^2(\R)$ defined by $[(g\partial)f](x):=g(x)f'(x)$ for $x\in\R$, generates a strongly continuous group $(G_{g\partial}(t),t\in\R)$ of bounded linear operators acting on the space $L^2(\R)$ that is given by 
	\begin{equation*}
		[G_{g\partial}f](x)=f\left(h^{-1}(h(x)+t)\right), \quad x\in\R, f\in L^2(\R),
	\end{equation*}
where $h:\R\rightarrow\R$ is the function defined by 
	\begin{equation*}
		h(x):= \int_c^x\frac{1}{g(r)}\d{r}, \quad x\in\R,
	\end{equation*}
for arbitrary fixed $c\in\R$.
\medskip\\
Although this fact seems to be well-known, cf. e.g. \cite[page 113]{DaPIanTub82}, we were unable to find the proof in our (operator semigroup) setting so we decided to include it here for the sake of completeness. It should be noted that this group of operators is closely connected to the translation equation (see e.g. \cite[section 1, chapter 6]{Acz66}), one-parameter Lie groups (see e.g. \cite[section 6 of chapter 8, p. 293 - 299]{Ham89}), and the theory of generalized shift operators (see e.g. \cite{LevLit89} and the many references therein).
\medskip\\
The proof is split into two lemmas. 
\begin{lemma}
\label{lem:A1}
The family of operators $(G_{g\partial}(t),t\in\R)$ is a strongly continuous group of bounded linear operators acting on the space $L^2(\R)$. 
\end{lemma}
\begin{proof}
We begin with some easy observations. Fix $c\in\R$. Since $g$ is continuous and bounded from below by the positive constant $c_1$, the function $h$ is well-defined and continuously differentiable with positive derivative, it is also increasing, and therefore injective. By the mean value theorem and the intermediate value theorem, $h$ is also surjective and, consequently, invertible on the whole real line and we denote its inverse by $h^{-1}$. Moreover, since $h$ is strictly monotone, the inverse $h^{-1}$ is continuous. Now, for every $t\in\R$, the operator $G_{g\partial}(t)$ is linear and bounded on $L^2(\R)$ by the following estimate:
	\begin{equation}
	\label{eq:estimate_group}
		\|G_{g\partial}(t)f\|_{L^2(\R)}^2 := \int_{\R} \left|f\left(h^{-1}(h(x)+t)\right)\right|^2\d{x} = \int_{\R} \left|f(x)\right|^2\frac{g\left(h^{-1}(h(x)-t)\right)}{g(x)}\d{x} \leq \frac{c_2}{c_1}\|f\|_{L^2(\R)}^2.
	\end{equation}
It is straightforward to verify that the system of bounded linear operators $(G_{g\partial}(t), t\in \R)$ is a group. Moreover, by estimate \eqref{eq:estimate_group}, we have that the function $G_{g\partial}: \R\rightarrow\mathscr{L}(L^2(\R))$ is uniformly bounded and hence, in order to show its strong continuity, it is enough to show that $G_{g\partial}(t)f$ tends to $f$ as $t\downarrow 0$ in the topology of $L^2(\R)$ for $f$ from a dense subset of $L^2(\R)$, cf. \cite[Exercise I.5.9 (5)]{EngNag00}. Let $f\in \mathscr{C}_c(\R)$. We have that 
	\begin{equation*}
		\lim_{t\downarrow 0} \|G_{g\partial}(t)f -f\|_{\infty} = \lim_{t\downarrow 0} \sup_{x\in\R}\left|f\left(h^{-1}(h(x)+t)\right) -f(x)\right| = 0
	\end{equation*}	
by (uniform) continuity of $f$ and continuity of $h^{-1}$ and $h$ on $\R$ which implies $G_{g(t)f}$ tends to $f$ in the topology of $L^2(\R)$.
\end{proof}
\begin{lemma}
\label{lem:A2}
The infinitesimal generator of the group $(G_{g\partial}(t), t\in\R)$ is the operator $g\partial$.
\end{lemma}
\begin{proof}
Let $u\in L^2(\R)$ and assume that there exists $z\in L^2(\R)$ such that 
	\begin{equation*}
		\lim_{s\downarrow 0}\left\|\frac{G_{g\partial}(s)u - u}{s} -z \right\|_{L^2(\R)}=0.
	\end{equation*}
Let $a\in\R$ be arbitrary and define a function $F:\R\rightarrow\R$ by
	\begin{equation*}	
		F(t):= \int_0^a[G_{g\partial}(r)u](t)\d{r}.
	\end{equation*}
Assume, for simplicity, that $a>0$. The case $a<0$ can be done in a very similar manner. By using the group property of $G_{g\partial}$, we have for $0<s<a$ that
	\begin{equation}
	\label{eq:generator_1}
		\frac{1}{s}\left( G_{g\partial}(s)F-F\right) = \frac{1}{s}\int_0^aG_{g\partial}(s+ r)u\d{r} - \frac{1}{s}\int_0^a G_{g\partial}(r)u\d{r}.
	\end{equation}
Consequently, we obtain that
	\begin{equation*}
		\frac{1}{s}\left( G_{g\partial}(s)F-F\right) = \int_0^a G_{g\partial}(r)\left[\frac{G_{g\partial}(s)u-u}{s} \right]\d{r} \qquad \overset{L^2(\R)}{\underset{s\downarrow 0}{\longrightarrow}} \qquad \int_0^a G_{g\partial}(r)z\d{r}
	\end{equation*}
where the convergence follows by the dominated convergence theorem and estimate \eqref{eq:estimate_group}. For this limiting function, there is the expression
	\begin{equation*}
		\int_0^a [G_{g\partial}(r)z](x)\d{r} = \int_0^a z\left(h^{-1}\left(h(x)+r\right)\right)\d{r} = \int_x^{h^{-1}(h(x)+a)} \frac{z(r)}{g(r)}\d{r}, \quad x\in\R. 
	\end{equation*}
On the other hand, we also have from equation \eqref{eq:generator_1} that
	\begin{align*}
		\frac{1}{s}\left( G_{g\partial}(s)F-F\right) &  = \frac{1}{s} \int_{s}^{s+a} G_{g\partial}(r)u\d{r} -\frac{1}{s}\int_0^aG_{g\partial}(r)u\d{r}\\
			& = \int_a^{a+s}G_{g\partial}(r)u\d{r} - \frac{1}{s}\int_0^sG_{g\partial}(r)u\d{r}
	\end{align*}
By strong continuity of the group $G_{g\partial}$ and Lebesgue's differentiation theorem, it follows that the right-hand side of the above equality tends to 	
	\begin{equation*}
		u\left(h^{-1}(h(\cdot)+a)\right) - u(\cdot)
	\end{equation*}
in $L^2(\R)$ as $s\downarrow 0$. Consequently, it follows that the inequality
	\begin{equation}
	\label{eq:generator_2}
		u\left(h^{-1}(h(x)+a)\right) - u(x) = \int_x^{h^{-1}(h(x)+a)}\frac{z(r)}{g(r)}\d{r}
	\end{equation}
is satisfied with almost every $x\in\R$. Finally, we have that for every $x, y\in \R$ there exists $a\in\R$ such that $y=h^{-1}(h(x)+a$ and therefore, possibly after modifying $u$ on a null set, we have that the equation 
	\begin{equation*}
		u(y)-u(x) = \int_x^y \frac{z(r)}{g(r)}\d{r}
	\end{equation*}
is satisfied for almost every $x,y\in\R$ by equality \eqref{eq:generator_2} which shows that $u\in W^{1,2}(\R)$ and it holds that $z=gu' = (g\partial)f$.
\end{proof}
\end{appendices}

\end{document}